\documentclass[a4paper, 10pt]{amsart}
\usepackage[utf8]{inputenc}
\usepackage[T1]{fontenc}
\usepackage{lmodern}
\usepackage{graphicx}
\usepackage[english]{babel}
\usepackage{subfigure}
\usepackage{mathrsfs}
\usepackage[bbgreekl]{mathbbol}
\usepackage{amsmath, amsfonts, amssymb, dsfont}
\usepackage{stmaryrd}
\usepackage{mathrsfs}
\usepackage{mathtools}
\usepackage{listings}
\usepackage{subfigure}
\usepackage{floatrow}
\usepackage{fourier-orns}
\usepackage{amsthm}
\usepackage{algorithm}
\usepackage{algorithmic}
\usepackage{textcomp}
\usepackage{faktor}
\usepackage{xfrac}
\usepackage[all]{xy}
\usepackage[colorlinks = True , allcolors = blue]{hyperref}

\newcommand{\tsp}[1]{{}^t \! #1}

\newcommand{\ii}{\mathrm{i}}

\newcommand{\Span}{\operatorname{Span}}

\newcommand*\diff{\mathop{}\!\mathrm{d}}

\newcommand{\Id}{\operatorname{Id}}

\newcommand{\R}{\mathbb R}
\newcommand{\N}{\mathbb N}

\newcommand{\Z}{\mathbb Z}

\newcommand{\CC}{\mathbb C}

\newcommand{\hol}{\operatorname{Hol}}
\newcommand{\CCC}{\mathscr{C}}
\newcommand{\T}{\mathcal T}
\newcommand{\F}{\mathcal{F}}

\newcommand{\End}{\operatorname{End}}
\newcommand{\Aut}{\operatorname{Aut}}
\newcommand{\Ad}{\operatorname{Ad}}

\newcommand{\g}{\mathfrak{g}}

\newcommand{\Heis}{\operatorname{Heis}}
\newcommand{\heis}{\mathfrak{heis}}

\newcommand{\vphi}{\varphi}

\newcommand{\act}{\curvearrowright}

\newcommand{\lag}{\langle}
\newcommand{\rag}{\rangle}

\newcommand{\op}{\operatorname{Op}}
\newcommand{\ttt}{\mathfrak{t}}
\newcommand{\toep}{\mathcal{T}}

\newcommand{\Sp}{\operatorname{Sp}}
\newcommand{\Sym}{\operatorname{Sym}}

\newcommand\isomto{\stackrel{\sim}{\smash{\longrightarrow}\rule{0pt}{0.4ex}}}

\def\dar[#1]{\ar@<2pt>[#1]\ar@<-2pt>[#1]}
\entrymodifiers={!!<0pt,0.7ex>+}

\newcommand{\eq}[1][r]
   {\ar@<-3pt>@{-}[#1]
    \ar@<-1pt>@{}[#1]|<{}="gauche"
    \ar@<+0pt>@{}[#1]|-{}="milieu"
    \ar@<+1pt>@{}[#1]|>{}="droite"
    \ar@/^2pt/@{-}"gauche";"milieu"
    \ar@/_2pt/@{-}"milieu";"droite"}
\everymath{\displaystyle\everymath{}}

\theoremstyle{plain}
\newtheorem{thm}{Theorem}[section]

\newtheorem{lem}[thm]{Lemma}
\newtheorem{prop}[thm]{Proposition}
\newtheorem{cor}{Corollary}[thm]
\newtheorem{ex}[thm]{Example}

\theoremstyle{definition}
\newtheorem{mydef}[thm]{Definition}

\theoremstyle{remark}
\newtheorem{rem}[thm]{Remark}

\DeclareSymbolFontAlphabet{\mathbb}{AMSb}
\DeclareSymbolFontAlphabet{\mathbbl}{bbold}

\title[]{Toeplitz algebras and the Heisenberg group}
\author{Clément Cren}
\address{Mathematisches Institut\\
Georg-August Universität Göttingen\\
Bunsenstraße 3-5\\
D-37073 Göttingen\\
Deutschland}
\email{\href{mailto:clement.cren@mathematik.uni-goettingen.de}{clement.cren@mathematik.uni-goettingen.de}}

\begin{document}

\begin{abstract}
We show an isomorphism between an algebra which is naturally constructed from the Toeplitz algebra generated by d-shifts, and an ideal of the $C^*$-algebra of the $(2d+1)$-dimensional Heisenberg group. 
This is a particular case of a more general result for graded nilpotent Lie groups involving symbols in the filtered calculus. The proof presented here however only involves basic functional analysis while still showcasing the ideas of the proof in the general setting.
\end{abstract}

\maketitle

\section{Introduction}

Two important algebraic structures in the mathematical formalism of quantum mechanics are the Heisenberg groups and the Toeplitz algebras. The Heisenberg groups appear when describing a quantum mechanical system. This system is understood through the representations of the group. At the level of the Lie algebra, these representation involve the so called creation and annihilation operators. One can construct such a group from any symplectic vector space (i.e. from any classical mechanical system). The collection of all representations, understood from the group \(C^*\)-algebra, gives a quantization of the symplectic vector space.

On the other hand, Toeplitz algebras appear in the Berezin-Toeplitz quantization. This is  another approach to quantization which produces explicit operators from classical observables. While this approach is equivalent to the representations of the Heisenberg group  (thanks to a theorem of Stone and von Neumann), it has the advantage of being described only using bounded operators, the \(d\)-shifts\footnote{Here \(d\) denotes the dimension of the mechanical system.}. These shifts and their adjoints have a well known relation to the creation and annihilation operators. In this article, we deepen this relation to define an isomorphism between two \(C^*\)-algebras, one constructed from the Heisenberg group and another one constructed from Toeplitz algebras.

More precisely, let \(d\in \mathbb N\) and denote by \(\mathcal{T}_d\) the Toeplitz algebra with \(d\)-generators and \(H_{2d+1}\) the Heisenberg group of dimension \(2d+1\). Let \(C^*(H_{2d+1})\) be the corresponding group \(C^*\)-algebra and \(C^*_0(H_{2d+1})\triangleleft C^*(H_{2d+1})\) the kernel of the trivial representation. We define an action of the real numbers \(\R \act \mathcal{T}_d\) using the complex powers of the so-called number operator. Recall that there is a quotient map \(\mathcal{T}_d \to \CCC(\mathbb{S}^{2d-1})\). We can use it to consider the fibered product of two Toeplitz algebras \(\mathcal{T}_d\bigoplus_{\CCC(\mathbb{S}^{2d-1})} \mathcal{T}_d\) on which \(\R\) still acts. We show in Theorem \ref{Main Isom} that there is an isomorphism:
\[\left(\mathcal{T}_d\bigoplus_{\CCC(\mathbb{S}^{2d-1})} \mathcal{T}_d\right) \rtimes \R \isomto C^*_0(H_{2d+1}).\]

To construct this isomorphism, we use the structure of \(C^*_0(H_{2d+1})\) as a continuous field of \(C^*\)-algebras over \(\R\) \cite{Lee}. The technical part is then in the construction of the map, to show that the shifts defining the Toeplitz algebra preserve the continuity of these sections at \(0\). This specific part involves Heisenberg pseudodifferential operators. Since the rest of the proof and the techniques used are more standard operator algebras constructions, we reserve these technicalities to a specific and separate section \ref{Symbols technicalities}. Indeed, we intend the rest of this paper to be readable with only general knowledge in \(C^*\)-algebras. This is why we devote Sections \ref{ToeplitzIntroduction} and \ref{HeisenbergIntroduction} to the constructions of the Toeplitz algebras, the Heisenberg group \(C^*\)-algebras, as well as their respective exact sequences.

We do a slight breach of this vow of simplicity and state our results for Heisenberg groups of symplectic vector spaces. This has the advantage of showing that all our constructions are canonical and of geometric nature. This geometric aspect allows us in Section \ref{ContactExtension} to generalize our construction to contact manifolds. On these manifolds we can construct a locally trivial bundle of Heisenberg groups. We show that our construction carries on to this case. Our construction depends however on a choice of a supplementary line bundle for the contact distribution (i.e. a Reeb field). This leads to a discussion on the co-orientability of the contact structure and its implications in the present work.

\section{Toeplitz algebras and the symmetric Fock space}\label{ToeplitzIntroduction}

In this section we describe the construction of the symmetric Fock space and the algebra of Toeplitz operators on it following \cite{Arveson}.

Let $V$ be a $d$-dimensional complex hermitian vector space. Denote by $T(V) = \bigoplus_{n = 0}^{+\infty}V^{\otimes n}$ its tensor algebra with the natural hermitian metric obtained from the one on $V$ and $\F(V)$ its completion for the corresponding norm. Denote by $\Sym(V) = \bigoplus_{n = 0}^{+\infty} \Sym^n(V) \subset \F(V)$ the subspace of symmetric tensors and $\F^+(V) \subset \F(V)$ its closure called the symmetric Fock space.
We denote by $\Sym \colon \F(V) \to \F(V)$ the orthogonal projector onto $\F^+(V)$.
For $v \in V$ and $n \in \N$ we write $v^n := v\otimes v \otimes \cdots \otimes v \in \Sym^n(V)$. Notice that $\Sym^n(V) = \Span\{v^n, v \in V\}$. More generally for vectors $v_1,\cdots,v_n \in V$ we write $v_1\cdots v_n = \Sym(v_1\otimes \cdots\otimes v_n)$.

\begin{prop}If $e_1,\cdots,e_d$ form an orthonormal basis of $V$ then $e_{i_1}\cdots e_{i_n}, 1\leq i_1 \leq \cdots \leq i_n \leq d $ forms an orthogonal basis of $\Sym^n(V)$. Equivalently we can rearrange this product to get the basis $e_1^{\alpha_1}\cdots e_d^{\alpha_d}, \sum_{j = 1}^d \alpha_j = n$. Moreover we have:
$$\| e_1^{\alpha_1}\cdots e_d^{\alpha_d} \|^2 =\frac{1}{n!}\prod_{j = 1}^d \alpha_j!.$$
Denote by $|\alpha_1,\cdots,\alpha_d\rag$ the corresponding normalized vectors.
\end{prop}

\begin{proof}
    The orthogonality is clear from the construction of the metric. To compute the norm, one needs to compute the inner product of \(e_1^{\alpha_1}\otimes \cdots \otimes e_d^{\alpha_d}\) against any any pure tensor obtained from permuting the term. This inner product is non zero (and equal to 1) if and only if the permutation only swaps each \(e_k\) with another \(e_k\). Taking the sum over all the permutations and averaging, we get the result.
\end{proof}

Consider for $1\leq j \leq d$ the shift operator:
$$S_j := \Sym(e_j \otimes \cdot) \colon \F^+(V) \to \F^+(V).$$

\begin{lem}For $1\leq j \leq d$ and $y_1,\cdots, y_n \in V$ we have:
\[S_j^*(y_1\cdots y_n) = \frac{1}{n}\sum_{k = 1}^n \lag y_k, e_j \rag y_1\cdots \hat{y_k}\cdots y_n.\]
Here \(\hat{y}_k\) means we omit this term in the tensor.\end{lem}

\begin{prop}Let $1\leq j \leq d$ and $\alpha \in \N^d$. We have the formulas:
$$S_j|\alpha_1,\cdots,\alpha_d\rag = \sqrt{\frac{\alpha_j + 1}{|\alpha| + 1}} |\alpha_1,\cdots, \alpha_j + 1,  \cdots ,\alpha_d\rag$$
$$S_j^*|\alpha_1,\cdots,\alpha_d\rag = \sqrt{\frac{\alpha_j}{|\alpha|}} |\alpha_1,\cdots, \alpha_j - 1,  \cdots ,\alpha_d\rag.$$
In the second line, we mean in particular that if $\alpha_j = 0$, then $S_j|\alpha\rag = 0$.
In particular the operators $S_j$ are injective and the $S_j^*$ surjective.
\end{prop}

\begin{proof}
    For \(\alpha\in \N^d\), we have \(S_j(e^{\alpha}) = e_1^{\alpha_1}\cdots e_j^{\alpha_j+1}\cdots e_d^{\alpha_d}.\) Normalizing on both sides gives the result.
\end{proof}

\begin{mydef}The $d$-dimensional Toeplitz algebra $\T(V)$ is the $C^*$-algebra of bounded operators on $\F^+(V)$ generated by the $S_j, 1\leq j \leq d$.\end{mydef}

In fact the Toeplitz algebra contains all the shift operators. Indeed if $v \in V$ then we can write $v = \sum_{j = 1}^d \lambda_j e_j$ and then $S_v = \Sym(v\otimes\cdot) = \sum_{j = 1}^d \lambda_j S_j$. Therefore the definition of $\T(V)$ does not depend on the choice of the basis.

Moreover, \(\T(V)\) is unital. We indeed have the relation 

\[\sum_{j = 1}^d S_j^*S_j |\alpha\rag = \frac{d + |\alpha|}{1+|\alpha|}|\alpha\rag, \alpha\in \N^d, \]
from which we can recover \(\Id\) using continuous functional calculus. We have \(\Id = f\left(\sum_{j = 1}^d S_j^*S_j\right)\) with \(f \colon x \mapsto \frac{d-x}{x-1}.\) Notice that for \(d = 1\), we directly have \(S_1^*S_1 = \Id\).

\begin{prop}If $U \colon V \to W$ is a unitary isomorphism. We can extend \(U\) to a map $\tilde{U} \colon \F^+(V) \to \F^+(W)$. This map is a unitary isomorphism  and intertwines the algebras $\T(V)$ and $\T(W)$. More precisely if $v \in V$ then $\tilde{U}\circ S_v = S_{U(v)}\circ \tilde{U}$. In particular we get a unitary isomorphism $\Ad(\tilde{U}) \colon \T(V) \isomto \T(W)$.
\end{prop}
\begin{proof}
Let $U$ be such an unitary map, $Ue_j = f_j, 1 \leq j \leq d$, $(f_j)_{1\leq j \leq d}$ being an orthonormal basis of a complex hermitian vector space $W$. Let $\tilde{U} \colon \F^+(V) \to \F^+(W)$ be the map defined by 
\[\forall v_1,\cdots,v_n\in V, \tilde{U}(v_1\cdots v_n) = (Uv_1)\cdots(Uv_n).\] 
This process applies to any linear map from \(V\) to \(W\) and preserves the composition. We also have $\tilde{U}^* = \widetilde{U^*}$ and so $\tilde{U}$ is unitary. If $S_j, S'_j$ denote the shift operators for the bases $(e_1,\cdots,e_d)$ and $(f_1,\cdots,f_d)$ respectively then we have $\forall 1 \leq j \leq d, \tilde{U}S_j\tilde{U}^* = S'_j$.
\end{proof}

\begin{lem}We have the inclusion $\mathcal{K}(\F^+(V)) \subset \T(V)$.\end{lem}
\begin{proof}
We fix a basis $e_1,\cdots,e_d \in V$ and let $S_1,\cdots, S_d$ be the corresponding shifts operators. For $1\leq i,j\leq d$ we then have the relations:
$$[S_j,S_i] = 0, [S_j^*,S_i] = (1+N)^{-1}(\delta_{i,j} - S_iS_j^*).$$
Here $N$ denotes the number operator. It is an unbounded operator which in the orthonormal basis $(|\alpha\rag)_{\alpha \in \N^d}$ reads:
$$\forall \alpha \in \N^d, N|\alpha\rag = \left(\sum_{j=1}^d \alpha_j\right)|\alpha\rag.$$
We have the relation $\sum_{j = 1}^d S_j^*S_j = (d+N)(1+N)^{-1}$ so $N$ can be recovered from the shifts using functional calculus.
Since $N$ is diagonal we clearly see that $(1+N)^{-1}$ is compact. Now since compact operators form a 2-sided closed ideal we get that $[\T(V),\T(V)] = \mathcal{K}(\F^+(V)) \subset \T(V)$.
\end{proof}

\begin{prop}\label{ToepExactSeq}There is an exact sequence:
$$\xymatrix{0\ar[r] & \mathcal{K}(\F^+(V)) \ar[r] & \T(V) \ar[r] & \CCC_0(\mathbb{S}^*V)\ar[r] & 0.}$$
\end{prop}
\begin{proof}
From the previous proof we have $[\T(V),\T(V)] = \mathcal{K}(\F^+(V))$. Therefore $\faktor{\T(V)}{\mathcal{K}(\F^+(V))}$ is an abelian $C^*$-algebra so by Gelfand duality it is an algebra of the type $\CCC(X)$ for some  compact Hausdorff space $X$ ($\T(V)$ is unital, which implies compactness of the spectrum). Let us take an orthonormal basis $e_1,\cdots,e_d \in V$ and denote by $S_1,\cdots,S_d$ the corresponding shift operators. The quotient is generated by the images $\overline{S_1},\cdots,\overline{S_d}$. These operators are mutually commuting and normal, so $X$ is their joint spectrum. As such, we can see \(X\) as a subset of $\CC^d$. Moreover, we have the relation 
\[\sum_{j = 1}^d \overline{S_j^*}\overline{S_j} = 1 + (d-1)\overline{(1+N)^{-1}} = 1,\] 
so $X \subset \mathbb{S}^{2d-1}$. Now this subset has to be non-empty. Let us show that the (transitive) action $\mathcal{U}(d) \act \mathbb{S}^{2d-1}$ preserves $X$. If $U \in \mathcal{U}(d)$ then from our choice of basis corresponds a unitary $U_1 \in \mathcal{U}(V)$. By the above functoriality results we get a unitary isomorphim $\tilde{U_1}$ of $\F^+(V)$ which preserves $\T(V)$. Now if $U = (u_{i,j})_{1\leq i,j\leq d}$ then we have:
$$\Ad(\tilde{U_1})S_i = \sum_{j = 1}^d \overline{u_{j,i}}S_j.$$
Therefore $X$ is stable under the action $\mathcal{U}(d) \act \mathbb{S}^{2d-1}$ and since it is non-empty and the action transitive we have $X=\mathbb{S}^{2d-1}$. The identification $\mathbb{S}^{2d-1} \cong \mathbb{S}^*V$ is also done through the choice of basis.
\end{proof}

\begin{rem}\label{QuotientToeplitz}Let $v \in V$ and denote by $\overline{S_v}$ the image of $S_v \in \T(V)$ in the quotient algebra. To $v$ we associate a function on $V^*$ using duality. This function is linear and in particular homogeneous of degree $1$. As such, it factors to a function $\tilde{v} \in \CCC(\mathbb{S}^*V)$. The reasoning in the proof defines a unitary isomorphism $\faktor{\T(V)}{\mathcal{K}(\F^+(V))} \isomto \CCC(\mathbb{S}^*V)$ by sending $\overline{S}_v$ to $\tilde{v}$. Through this isomorphism the operator $\overline{S_v^*}$ is sent to $\overline{\tilde{v}}$, the pointwise complex conjugate of the function $\tilde{v}$.\end{rem}

\section[Heisenberg groups and their C-algebras]{Heisenberg groups and their \texorpdfstring{$C^*-$}-algebras}\label{HeisenbergIntroduction}

Let $(V,\omega)$ be a real symplectic vector space of dimension $2d$. Define its corresponding Heisenberg Lie algebra $\heis(V,\omega)$ as the vector space $V\oplus \R$ endowed with Lie bracket given by $[v, w] = \omega(v,w)Z$, where $Z$ is a basis of $\R$. If $(e_j,f_j)_{1\leq j \leq d}$ is a symplectic basis for $(V,\omega)$ then we have $[e_j,f_k] = \delta_{j,k}Z$ and the other brackets vanish. We denote by $\Heis(V,\omega)$ the corresponding connected, simply connected Lie group. Since $\heis(V,\omega)$ is nilpotent, we can use it as the underlying space for the group as well and the product is defined through the Baker-Campbell-Haussdorff formula.

To understand the structure of its $C^*$-algebra we need to understand the unitary irreducible representations of the group. Equivalently they correspond to representations of the Lie algebra $\heis(V,\omega)$ on a complex Hilbert space by anti-self-adjoint unbounded operators. We first have the characters, parameterized by $V^*$, which correspond to the representations $\pi \in \widehat{\Heis(V,\omega)}$ with $\pi(Z) = 0$.
Now if $\pi(Z) \neq 0$, a theorem of Stone and von Neumann states that we necessarily have $\pi(Z) = \ii \lambda \Id, \lambda \in \R$ and for a given $\lambda$ there is only one such unitary irreducible reprensentation up to conjugation (moreover the underlying Hilbert space is necessarily infinite dimensional), we denote it by $\pi_{\lambda}$ (indifferently for the group, the Lie algebra or the group $C^*$-algebra). We give an explicit construction of this representation using the symmetric Fock space.

Let $J \colon V \to V$ be a compatible complex structure, i.e. $J^2 = -\Id$, $\omega(J\cdot,\cdot)$ is positive definite and $\omega(J\cdot,J\cdot) = \omega$. In a symplectic basis as before, a choice of compatible complex structure would be $Je_j = -f_j, Jf_j =  e_j$.
Consider the complexified vector space $V\otimes \CC$ and $V^{1,0}, V^{0,1}$ the eigenspaces of \(J\) for $\ii$ and $-\ii$ respectively. We endow $V^{1,0}$ with the Hermitian metric $\lag\cdot,\cdot\rag = \ii\omega(\cdot,\bar{\cdot})$ restricts to a hermitian metric on $V^{1,0}$. From this we get the symmetric Fock space $\F^+(V^{1,0})$ as in the previous section. Now let $W_1,\cdots, W_d$ be an orthonormal basis for $V^{1,0}$. We have $\heis(V,\omega) \otimes \CC = V^{1,0} \oplus V^{0,1} \oplus \CC Z$. It is spanned as a vector space by $W_j, \bar{W_j}, Z, 1 \leq j \leq d$ and we have the relations $[W_j,\bar{W_j}] = -\ii Z$, the other Lie brackets being equal to $0$. To construct $\pi_{\lambda}, \lambda > 0$ we just need to construct $\pi = \pi_1$ and then set $\pi_{\lambda}(W_j) = \sqrt{\lambda}\pi(W_j), \pi_{\lambda}(\bar{W_j}) = \sqrt{\lambda}\pi(\bar{W_j}), 1 \leq j \leq d$. This indeed ensures that $\pi_{\lambda}(Z) = \ii \lambda \Id$ as long as the relation is satisfied for \(\lambda = 1\).
Likewise we can obtain every $\pi_{\lambda}$ with $\lambda < 0$ from $\pi_{-1}$. The later can itself be obtained from $\pi_1$ for the group $\Heis(V,-\omega)$. This manipulation reverses the orientation of the symplectic vector space, and in particular reverses the roles of $V^{1,0}$ and $V^{0,1}$, we would then be acting on $\F^+(V^{0,1})$.

Now for the representation $\pi$, the operators $\pi(W_j), \pi(\bar{W_j})$ satisfy the canonical commutation relations: $[\pi(W_j),\pi(\bar{W_j})] = \pi(-\ii Z) = \Id$. They are thus related to the so called creation and annihilation operators. Denote by $A_j$ the annihilation operator with $1\leq j \leq d$. Since for every $1\leq j \leq d$ we need the relations $\pi(\overline{W_j}) = -\pi(W_j)^*$ then we set $\pi(\overline{W_j}) = \ii A_j$ and $\pi(W_j) = \ii A_j^*$. The creation and annihilation operators are defined as follows:
\begin{align*}
A_j|\alpha\rag &= \sqrt{\alpha_j}|\alpha_1,\cdots,\alpha_j-1,\cdots,\alpha_d\rag\\
A^*_j|\alpha\rag &= \sqrt{\alpha_j+1}|\alpha_1,\cdots,\alpha_j+1,\cdots,\alpha_d\rag
\end{align*}

One can check that we indeed have the canonical commutation relations: 
\[\forall 1 \leq j,k\leq d, [A_j,A_k^*] = \delta_{j,k}\Id.\] 
To understand the previous equations it is interesting to compare our model of the symmetric Fock space with the Bargmann space. Remember that $\F^+(V^{1,0})$ is the completion of the algebraic sum \(\bigoplus_{k\in \N}\Sym^k(\CC^d)\). The later can be seen as the space of polynomial functions on $V^{0,1} \cong \left(V^{1,0}\right)^*$. The symmetric Fock space then becomes a space of weighted-\(L^2\) holomorphic functions on the space $V^{0,1}$ with the hermitian product defined for $f,g\in \hol(V^{0,1})$ as:
$$\lag f, g \rag := \frac{1}{(2\pi)^d}\int_{V^{0,1}} f(z)\overline{g(z)}\diff z.$$
An orthogonal basis for this space is given by the monomials, once normalized it becomes: 
$$\left( \frac{1}{\sqrt{\alpha_1!\cdots\alpha_d!}}z_1^{\alpha_1}\cdots z_d^{\alpha_d}\right)_{\alpha \in \N^d}.$$
We will use the multi-index notations $\alpha! = \prod_{j = 1}^d \alpha_j!$ and $z^{\alpha} = \prod_{j = 1}^d z_j^{\alpha_j}$ for $\alpha \in \N^d$ and $z \in V^{0,1}$. On this space of holomorphic functions\footnote{Notice that we dont get \emph{all} the entire functions on \(V^{0,1}\). Write an entire function \(f\) as a power series \(f(z) = \sum_{\alpha\in \N^d}\frac{c_{\alpha}}{\alpha!}z^{\alpha}\). Then \(f\) belongs to the Bargmann space if and only if \((c_{\alpha})_{\alpha\in \N^d}\in \ell^2(\N^d)\).}, the respective creation and annihilation operators are $\widetilde{A}_j, \widetilde{A}_j^*, 1 \leq j \leq d$ defined by:
$$\forall 1 \leq j \leq d, \widetilde{A}_j = \ii\frac{\partial}{\partial z_j}, \widetilde{A}^*_j = \ii z_j.$$
Now there is a unitary isomorphism between the symmetric Fock space and the Bargmann space on $V^{1,0}$, namely the one sending $|\alpha\rag$ to $\frac{1}{\sqrt{\alpha!}}z^{\alpha}$ for all $\alpha \in \N^d$. If we denote this unitary isomorphism by $T \colon \F^+(V^{1,0}) \to \hol(V^{0,1})$ (the Bargmann transform) then we have: 
$$\forall \alpha \in \N^d, T(e_1^{\alpha_1}\cdots e_d^{\alpha_d}) = \frac{1}{\sqrt{|\alpha|!}}z_1^{\alpha_1}\cdots z_d^{\alpha_d}, \alpha \in \N^d,$$
and thus
$$\forall 1\leq j \leq d, T\widetilde{A}_jT^* = A_j.$$

\begin{rem}
    The model of the Bargmann space has seen extensive use in complex geometry and geometric quantization. The Berezin-Toeplitz quantization can be performed this way on any compact Kähler manifold, see \cite{LeFloch}.
\end{rem}

We have thus far obtained a representation of the Lie algebra $\heis(V,\omega)$ by (closed) unbounded anti-hermitian operators. It lifts as a representation:
\[\pi \colon \Heis(V,\omega) \to \mathcal{U}(\F^+(V^{1,0})).\]

\begin{prop}$\pi(C^*(\Heis(V,\omega))) = \mathcal{K}(\F^+(V^{1,0}))$.\end{prop}

Up to this point we have described the unitary dual of the Heisenberg group as a set. We have identified:
$$\widehat{\Heis(V,\omega)} = V^* \sqcup \R^*.$$
Here the $V^*$ component corresponds to the characters and the $\R^*$ component corresponds to the infinite dimensional representations. To describe the group $C^*$-algebra we also need to understand the topology of the spectrum. This can be done using Kirillov theory \cite{Kirillov}. Kirillov shows that, for a simply connected nilpotent Lie group, there is a homeomorphism between the unitary dual (with the hull-kernel/Fell topology) and the space of coadjoint orbits (with the quotient of the euclidean topology):
$$\widehat{G} \cong \faktor{\g}{\Ad^*(G)}.$$
The explicit description of this homeomorphism does not matter for our purpose. We can however compute the space of coadjoint orbits for $\Heis(V,\omega)$. We have $\heis(V,\omega) = V^* \oplus \R Z^*$ and the orbits are given by the hyperplanes $\{Z^* = \lambda\}, \lambda \neq 0$ and the points $\{(\eta,0)\}, \eta \in V^*$. Through Kirillov homeomorphism the hyperplane $\{Z^* = \lambda\}, \lambda \neq 0$ corresponds to the representation infinite dimensional representation $\pi_{\lambda}$ described above and the point $\{(\eta,0)\}, \eta \in V^*$ to the character with same element $\eta$. 

The topology of the unitary dual is now clear from Kirillov's homeomorphism. The subsets $V^*$ and $\R^*$ are endowed with their usual topology and the later is open. Moreover a sequence of points (or a net) in $\R^*$ converging to $0$ for the usual topology converges in $\widehat{\Heis(V,\omega)}$ to every point in the subset $V^*$.
In particular we have an open map $\widehat{\Heis(V,\omega)} \to \R$ sending the whole $V^*$ to $0$. A theorem of Lee \cite{Lee} implies the following structure for the group $C^*$-algebra:

\begin{prop}$C^*(\Heis(V,\omega))$ is a continuous field of $C^*$-algebras over $\R$. The fiber at $\lambda > 0$ is given by $\mathcal{K}(\F^+(V^{1,0}))$, for $\lambda <0$ it is given by $\mathcal{K}(\F^+(V^{0,1}))$ and for $\lambda = 0$ by $\CCC_0(V^*)$.
In particular, by restricting a section to the $0$-fiber, we obtain the exact sequence:
\[\xymatrix@C-1pc{0 \ar[r] & \CCC_0(\R^*_+,\mathcal{K}(\F^+(V^{1,0}))\oplus \mathcal{K}(\F^+(V^{0,1})))  \ar[r] & C^*(\Heis(V,\omega)) \ar[r] & \CCC_0(V^*) \ar[r] & 0}.\]
\end{prop}

\begin{rem} A more algebraic way of seeing the fibers of the $C^*$-algebra can be done using twisted convolution $C^*$-algebras. The $2$-form $\omega$ provides a $2$-cocycle on the abelian group $V$ and we can identify the algebra $\pi_{\lambda}(C^*(V,\omega)), \lambda \neq 0$ with the twisted convolution algebra $C^*(V,\lambda\omega)$. This identification still holds for $\lambda = 0$, on the one hand, we obtain the zero-fiber of $C^*(\Heis(V,\omega))$ seen as a continuous field $C^*$-algebra. On the other hand, the cocycle vanishes and we get the non-twisted convolution algebra which we can identify, since $V$ is abelian, to $\CCC_0(V^*)$. This description also highlights the change of orientation of the symplectic vector space when going from positively indexed representations to the negatively indexed ones.
\end{rem}

On the Heisenberg group there is a natural group of automorphisms : the inhomogeneous dilations. These morphisms are defined at the level of Lie algebra by 
$$\forall \lambda > 0, \forall (v,t) \in \heis(V,\omega) = V \oplus \R, \delta_{\lambda}(v,t) = (\lambda v, \lambda^2 t).$$
These linear maps are compatible with the Lie bracket and can thus be lifted to Lie algebra automorphisms of $\Heis(V,\omega)$ (still denoted by $\delta_{\lambda}, \lambda>0$). We thus obtain by precomposition an action $\R^*_+ \act \widehat{\Heis(V,\omega)}$. Through Kirillov's homeomorphism, this map corresponds to the quotient of the dual action given by the $\tsp\delta_{\lambda}, \lambda >0$ on $\heis(V,\omega)^*$. The action of a $\lambda >0$ is thus the multiplication by $\lambda^2$ on $\R^* \subset \widehat{\Heis(V,\omega)}$ and the multiplication by $\lambda$ on $V^* \subset \widehat{\Heis(V,\omega)}$. In particular the trivial representation is the only fixed point.
If we consider the kernel of the trivial representation $C^*_0(\Heis(V,\omega)) \triangleleft C^*(\Heis(V,\omega))$, we get a new $C^*$-algebra, stable under the action $\delta$ and whose corresponding action on the spectrum is free. One would like to have $C^*_0(\Heis(V,\omega)) = A \rtimes \R$ so that $C^*_0(\Heis(V,\omega))\rtimes_{\delta}\R^* \cong A \rtimes \R \rtimes \R^* \cong A \otimes \mathcal{K}(L^2(\R))$. We propose to construct this algebra $A$ using the Toeplitz algebras.
Notice first that $C^*_0(\Heis(V,\omega))$ sits in the following exact sequence:
\[\xymatrix@C-1.25pc{0 \ar[r] & \CCC_0(\R^*_+,\mathcal{K}(\F^+(V^{1,0}))\oplus \mathcal{K}(\F^+(V^{0,1})))  \ar[r] & C^*_0(\Heis(V,\omega)) \ar[r] & \CCC_0(V^*\setminus \{0\}) \ar[r] & 0}\]

From this exact sequence, we can decompose the $C^*$-algebra $C^*_0(\Heis(V,\omega))$ into two parts $I_{\pm}(V,\omega)$ defined respectively as the sections over $\R_{\pm}$ (they are quotients of \(C^*_0(\Heis(V,\omega))\). The algebra is then isomorphic to the fibered sum:
$$C^*_0(\Heis(V,\omega)) = I_+(V,\omega) \bigoplus_{\CCC_0(V^*\setminus \{0\})} I_-(V,\omega).$$
These ideals are still stable under the $\delta$ action. Moreover, both ideals still sit into similar exact sequences\footnote{The one for $I_-(V,\omega)$ is similar, we just have to replace $V^{1,0}$ with $V^{0,1}$, which corresponds again to a change of orientation (we replace \(\omega\) by \(-\omega\)).}:
$$\xymatrix{0 \ar[r] & \CCC_0(\R^*_+,\mathcal{K}(\F^+(V^{1,0}))) \ar[r] & I_+(V,\omega) \ar[r] & \CCC_0(V^*\setminus \{0\}) \ar[r] & 0}.$$

\section{Isomorphism for d=1}

We first explain the isomorphism for $d = 1$. The construction of the crossed product is exactly the same but to get the isomorphism, we can use some nice arguments that only work for \(d = 1\) so we treat this case separately. Indeed, for \(d = 1\), the Toeplitz algebra $\T := C^*(S) = \T_1$ has the following universal property:

\begin{thm}[\cite{Coburn}] Let $A$ be a $C^*$-algebra and $v \in A$ a partial isometry. Then there exists a unique $*$-homomorphism $\varphi \colon \T \to A$ sending $S$ to $v$.\end{thm}

We wish to find an action $\R\act \T$ and a $*$-homomorphisms $\T \rtimes \R \to I_{\pm}(\R^2,\omega_{st})$ in a way that preserves the exact sequences and show that these morphisms are isomorphisms.
%We first define the action. Recall that we can see $\T$ as a subalgebra bounded operators on $\hol(\CC)$ or $\F^+(\CC) = \ell^2(\N)$. On this space of holomorphic function there is a quantum harmonic oscillator. This is an operator that is unbounded self-adjoint and can be expressed in terms of the creation and annihilation operators:
%$$H = AA^* + \frac{1}{2}\Id.$$
%In particular on $\ell^2(\N)$, the creation and annihilation operators are defined as $\forall n\in \N, Ae_n = \ii \sqrt{n}e_{n-1}, A^*e_n = \ii\sqrt{n+1}e_{n+1}$ (with the convention that $e_{-1} = 0$). Therefore the quantum harmonic oscillator becomes:
%$$\forall n\in N, He_n = \left(n+\frac{1}{2}\right)e_n.$$
%The quantum harmonic oscillator is positive and self-adjoint. We can use functional calculus to define its complex powers $(H^{\ii t})_{t\in \R}$. 
%Let $\alpha \colon \R \to \Aut(\T)$ be the family of automorphisms such that $\alpha_t$ sends $S$ to $S_t = \Ad(H^{\ii t})$. It is not clear yet that $\alpha_t$ are indeed automorphisms of $\T$. Let us show each $\alpha_t$ acts trivially on the ideal and the quotient. We begin with the following lemma:

We first define the action. We have \(\mathcal{F}^+(\CC) = \ell^2(\N)\). Consider the number operator \(N \colon \ell^2(\N) \to \ell^2(\N)\) by \( N(e_n)= n e_n, n \in \N\). The operator \(N+1\) is self adjoint and positive, we can consider its complex powers \((N+1)^{\ii t}, t\in \R\) using functional calculus.
Let $\alpha \colon \R \to \Aut(\T)$ be the family of automorphisms such that $\alpha_t$ sends $S$ to $S_t = \Ad((N+1)^{\ii t/2})(S)$. It is not clear yet that $\alpha_t$ are indeed automorphisms of $\T$. Let us show each $\alpha_t$ acts trivially on the ideal and the quotient. We begin with the following lemma:

\begin{lem}\label{WellDefd1}$\forall t \in \R, S_t \equiv S \mod \mathcal{K}$ \end{lem}

The lemma implies that the action is well defined and that it is trivial on the quotient.
\begin{proof}
Notice first that the action $t\mapsto\Ad((N+1)^{\ii t/2})$ automatically preserves the ideal of compact operators. Let us compute what happens with the action on finite rank operators.
For $n,m \in \N$, let $P_{n,m}$ be the rank one operator sending $e_n$ to $e_m$. The operator $\alpha_t(P_{n,m})$ is still of rank one and sends $\CC  e_n$ to $\CC e_m$. More precisely:
\begin{align*}
\alpha_t(P_{n,m})(e_n) &= S_t^mP_{0,0}S_t^{*n}e_n \\
					   &= (N+1)^{\ii t/2} S^m P_{0,0} S^{*n} (N+1)^{-\ii t/2} e_n \\
					   &=\left(\frac{m+1}{n+1}\right)^{\ii t/2} e_m.
\end{align*}
We now prove the lemma.
Define $\Psi_t = S_t - S$ and it's finite rank truncations:
$$\Psi_t^{(n)} \colon e_k \mapsto \left\lbrace \begin{matrix}
\Psi_t(e_k) \text{ if } k \leq n \\
0 \text{ otherwise.}
\end{matrix}\right.$$
The $\Psi_t^{(n)}$ are of course compact operators so we now show that:
$$\lim_{n\to +\infty}\Psi_t^{(n)} = \Psi_t$$
in the operator norm topology. Let $\Phi_t^{(n)} =  \Psi_t - \Psi_t^{(n)} $ we have:
$$\|\Phi_t^{(n)}\|^2 = \|\Phi_t^{(n)*}\Phi_t^{(n)}\| = \rho(\Phi_t^{(n)*}\Phi_t^{(n)}),$$
where $\rho$ is the spectral radius. We can compute the eigenvalues explicitly:
$$\Phi_t^{(n)*}\Phi_t^{(n)}(e_k) = \left\lbrace \begin{matrix}
\left|\left(\frac{k + 2}{k + 1} \right)^{\ii t/2} - 1 \right|^2 e_k \text{ if } k > n \\
0 \text{ otherwise.}
\end{matrix}\right.$$
Therefore we have 
$$\|\Psi_t^{(n)} - \Psi_t\| = \sup_{k \geq n} \left|\left(\frac{k+2}{k + 1} \right)^{\ii t/2} - 1 \right|.$$
We then get the approximation:
\begin{align*}
\left|\left(\frac{k + 2}{k + 1} \right)^{\ii t/2} - 1 \right| &= \left|\exp\left(\frac{\ii t}{2} \left[\ln\left(1 + \frac{2}{k}\right) - \ln\left(1 + \frac{1}{k}\right) \right] \right) - 1 \right| \\
  &= \left|\exp\left( \frac{\ii t}{2k} + o\left(\frac{1}{k}\right)\right) - 1\right| \\
  &= \mathcal{O}\left(\frac{1}{k}\right).
\end{align*}
Here the $o$ and $\mathcal{O}$ can be taken uniformly on $t$ in any compact subset of $\R$. Hence we get the convergence:
$$\lim_{n\to +\infty}\Psi_t^{\left(n\right)} = \Psi_t,$$
in the norm operator topology (uniformly for $t$ in any compact subset of $\R$). Thus the operators $\Psi_t$ are compact and we have proven the lemma.\end{proof}

\begin{cor}
The $\R$ action on $\T$ is well defined and trivial on the quotient $\CCC(\mathbb{S}^1)$.
\end{cor}

\begin{lem}
The crossed product $\mathcal{K}\rtimes_{\alpha} \R$ is trivial i.e. isomorphic to the tensor product $\mathcal{K}\otimes \CCC_0(\R^*_+)$.
\end{lem}

\begin{proof}
We have $\alpha_t = Ad((N+1)^{\ii t/2})$ and the $(N+1)^{\ii t/2}$ are multipliers of $\mathcal{K}$ so we can define the map :
$$\begin{matrix}
\mathcal{K}\otimes_{alg} \CCC_c(\R) & \to & \CCC_c(\R,\mathcal{K}) \\
f \otimes T & \mapsto & (s \mapsto f(s)T(N+1)^{-\ii s/2})
\end{matrix}$$
which is an algebra morphism (the later is considered with the product twisted by $\alpha$) extends to an isomorphism between $\mathcal{K}\otimes \CCC_0(\R^*_+)$ and $\mathcal{K}\rtimes_{\alpha} \R$
\end{proof}

\begin{rem}We used the fact that $\R$ is amenable hence we did not make any difference between the maximal and reduced crossed products by $\R$ and tensor products by $C^*(\R)$. We also used the Pontryagin duality: 
$$C^*(\R) \cong \CCC_0(\R^*_+)^.$$
\end{rem}

\begin{cor} Let \(A := \T \bigoplus_{\CCC(\mathbb{S}^1)}\T\). We have the exact sequence
\begin{equation} \label{DoubleCrossedToeplitz} \xymatrix{0 \ar[r] & \CCC_0(\R^*_+,\mathcal{K}\oplus\mathcal{K})  \ar[r] & A \rtimes_{\alpha} \R \ar[r] & \CCC_0(\R^2\setminus 0) \ar[r] & 0.}
\end{equation}
\end{cor}

\begin{proof}We use the exactness of the maximal crossed product (we can also use the reduced tensor product which is exact since $\R$ is amenable). We have already identified the ideal in the previous lemma. The action on the quotient being trivial we have 
$$\CCC(\mathbb{S}^1)\rtimes_{\alpha}\R \cong \CCC(\mathbb{S}^1)\otimes \CCC_0(\R^*_+) \cong \CCC_0(\mathbb{S}^1 \times \R^*_+)\cong \CCC_0(\R^2\setminus 0).$$
Therefore the exact sequence \eqref{DoubleCrossedToeplitz} derives from \ref{ToepExactSeq}.
\end{proof}

\begin{rem}Looking at $\alpha$ acting on $\toep$ we also get the exact sequence:
\begin{equation}\label{CrossedToeplitz}
\xymatrix{0 \ar[r] & \CCC_0(\R^*_+,\mathcal{K})  \ar[r] & \toep \rtimes_{\alpha} \R \ar[r] & \CCC_0(\R^2\setminus 0) \ar[r] & 0.}
\end{equation}
We still have $A \rtimes \R \cong (\toep \rtimes_\alpha \R) \oplus_{\CCC_0(\R^2\setminus 0)} (\toep \rtimes_\alpha \R)$.
\end{rem}

We now want to relate this exact sequence to the respective ones of $I_{\pm} = I_{\pm}(\R^2,\omega_{st})$. We define maps $\T \rtimes \R \to I_+$, show that they preserve the respective exact sequences and are isomorphisms.
For that we need two morphisms: $\vphi_{\pm} \colon \T \to M(I_\pm)$ and $\beta \colon \R \to M_u(I_\pm)$ ($M(\cdot)$ denotes the multiplier algebra and $M_u(\cdot)$ the unitary elements of this algebra). Those morphisms have to satisfy the relation: 
\begin{equation}\label{Equiv}
\forall t\in \R, \vphi \circ \alpha_t = \Ad(\beta(t)) \circ \vphi.
\end{equation}

We only give the idea behind such maps. Proving that they are well defined will actually require the same tools as those developed for the general case in the Sections \ref{IsomorphismSection} and \ref{Symbols technicalities}. 

Let \(W = \frac{1}{\ii\sqrt{2}} (X+\ii Y) \in \mathfrak{h}_3\otimes \CC\). We can consider it as a differential operator on the Heisenberg group, i.e. a map on \(\CCC^{\infty}_c(H_3)\). We still call \(W\) the restriction of \(W\) to an unbounded operator on \(I_+\).
Consider the polar decomposition \(W := u|W|\). We claim that \(u \in M(I_+)\) and that it is an isometry. In particular using Coburn's theorem we get the map \(\varphi_+\).

The intuition behind this choice can be understood in the irreducible representations of \(I_+\). Remember that \(I_+\) is a continuous field of \(C^*\)-algebras over \(\R_+\). Let \(\pi\) be the representation corresponding to the fiber over \(1\) (we get all the other positive ones using the dilations). Then \(\pi(W) = A^*\) is the creation operator. In particular its polar decomposition is \(A^* = S(N+1)^{-1/2}\), therefore the mapping \(u\mapsto S\) at least makes sense in each infinite dimensional irreducible representation. The remaining thing to show is that this has a continuous limit when the parameter of the representation tends to \(0\). This is a particular instance of Theorem \ref{Toeplitz as Heisenberg multipliers} whose proof will require the tools of Section \ref{Symbols technicalities}.

The construction of the \(\R\)-group of unitary multipliers is similar. We use the complex powers of the operator \(W^*W\). In all the irreducible representations of \(I_+\) this corresponds to \(N+1\). We claim that \(W^*W\) is strictly positive as an unbounded multiplier of \(I_+\) so that these complex powers make sense.

We can do the same thing on \(I_-\) by replacing \(W\) by \(\bar{W}\). The desired relations \ref{Equiv} are then easily satisfied.

\begin{rem}\label{Remarque:H ou N+1}
    We did not have to necessarily use conjugation by \(N+1\) to define the \(\R\) action. We could have for instance replaced \(N+1\) by the quantum Harmonic oscillator \(H = N+\frac{1}{2}\) which can also be realized as a differential operator on the Heisenberg group (which we then use to define the family of unitary multipliers). Which operators can be chosen will be more apparent in Section \ref{Symbols technicalities}.
\end{rem}

\section{Isomorphism for \texorpdfstring{$d$}{d} arbitrary}\label{IsomorphismSection}

\subsection{The \texorpdfstring{$\R$}{R}-action on the Toeplitz algebra}

We now consider an arbitrary symplectic vector space $(V,\omega)$ of dimension $2d$ for which we fix a compatible complex structure $J$. We want to define an isomorphism:
$$\left(\T(V^{1,0})\bigoplus_{\CCC(\mathbb{S}^*V)} \T(V^{0,1})\right)\rtimes \R \isomto C^*_0(\Heis(V,\omega)).$$
Here we understand the identifications $\CCC(\mathbb{S}^*V) \cong \CCC(\mathbb{S}^*V^{1,0})$ through $V \cong V^{1,0}$ with $V^{1,0} = \ker(J-\ii) = \{X - \ii JX, X \in V\}$. We get a similar identification for $V^{0,1}$. To get the desired isomorphism we can reduce ourselves to show an isomorphism:
$$\T(V^{1,0}) \rtimes \R \isomto I_+(V,\omega).$$

\begin{rem}[On the choice of compatible complex structure]Let $J,J'$ be two compatible complex structures on $(V,\omega)$. We can find an element $\varphi \in \Sp(V,\omega)$ such that $\varphi J\varphi^{-1}= J'$. On the complexified spaces, $\varphi$ preserves the eigenvalue decompositions and becomes a unitary operator between them, i.e. $\varphi = \varphi_{1,0}+\varphi_{0,1}$ with $\varphi_{1,0} \colon V^{1,0;J} \isomto V^{1,0;J'}$ and $\varphi_{0,1} \colon V^{0,1;J} \isomto V^{0,1;J'}$ both unitary. Lifting both these operators to the symmetric Fock spaces we get unitary isomorphisms $\Ad(\tilde{\varphi}_{1,0}) \colon \T(V^{1,0;J}) \isomto \T(V^{1,0;J'})$ and $\Ad(\tilde{\varphi}_{0,1}) \colon \T(V^{0,1;J}) \isomto \T(V^{0,1;J'})$. Moreover if we take another $\phi \in \Sp(V,\omega)$ intertwining both complex structures then $\varphi^{-1}\phi$ preserves $J$ and the same goes for $\varphi\phi^{-1}$ and $J'$. We then have the commutative diagram:
$$\xymatrix{\T(V^{1,0;J}) \ar[r]^{\Ad(\tilde{\varphi}_{1,0})} \ar[d]_{\Ad(\widetilde{(\phi^{-1}\varphi)}_{1,0})} & \T(V^{1,0;J'}) \ar[d]^{\Id} \\
\T(V^{1,0;J})\ar[r]^{\Ad(\tilde{\phi}_{1,0})} & \T(V^{1,0;J'}).}$$
And a similar diagram for the $(0,1)$-components. Therefore the algebras $\T(V^{1,0})$ and $\T(V^{0,1})$ are unique up to unique isomorphism and depend only on the symplectic structure.\end{rem}

\begin{rem}[Opposite structures] The space \(V^{0,1}\) is obtained from \(V^{1,0}\) being its complex conjugate. At the level of Toeplitz algebras this gives an isomorphism: \(\T(V^{1,0}) \cong \T(V^{0,1})^{op}\).

Now, \(V^{0,1}\) can also be seen as the \(+\ii\) eigenspace for the complex structure \(-J\). This new complex structure is not compatible with \(\omega\) but with \(-\omega\). Therefore, if we replace $\omega$ by $-\omega$ we swap the roles of $V^{1,0}$ and $V^{0,1}$ and thus switch the algebras $\T(V^{1,0})$ and $\T(V^{0,1})$. 

At the level of the Heisenberg group we have \(\Heis(V,-\omega) \cong \Heis(V,\omega)^{op}\) (and similarly for the corresponding \(C^*\)-algebras). Combining this with the isomorphism between the Toeplitz algebras with opposite complex structure we see that our isomorphism behaves well when taking the opposite symplectic structure.
\end{rem}

Before defining our isomorphism, let us first define the $\R$-action on the Toeplitz algebra. Let $W$ be a complex vector space. We once again consider the number operator \(N\) and we add 1 to get the injective \(N+1\). It is an unbounded self-adjoint operator. On the component $\Sym^k(W) \subset \F^+(W)$ it restricts to $(k+1)\Id$. In particular we can define the operator $(N+1)^{\ii t/2}$ for $t\in \R$ using functional calculus. We have $(N+1)^{\ii t/2}\in \mathcal{B}(\F^+(W)), t\in \R$ (and continuity in $t$ for the strong operator topology). Since $\mathcal{K}(\F^+(W))$ is a 2-sided ideal it is preserved by each $\Ad((N+1)^{\ii t/2})$. Now we want to prove that it preserves the whole Toeplitz algebra.

\begin{lem}For $w \in W$ we have $\Ad((N+1)^{\ii t/2})(S_w) \equiv S_w \mod \mathcal{K}(\F^+(W))$. In particular $t\mapsto\Ad((N+1)^{\ii t/2})$ defines an action $\R \act \T(W)$.\end{lem}
\begin{proof}
The proof is similar to Lemma \ref{WellDefd1}. It is enough to prove the result for $S_1,\cdots , S_d$ given an orthonormal basis $e_1,\cdots,e_d$ of $W$. Let us fix $1\leq j \leq d$ and set $\Psi_t = \Ad((N+1)^{\ii t/2})(S_j) - S_j$. 
We have $\Psi_t^*\Psi_t(|\alpha\rag) = \left(\left(\frac{|\alpha| + 2}{|\alpha| + 1}\right)^{\ii t/2} - 1\right)S_j|\alpha\rag$.
Let $N\in \N$ and set:
$$\Psi_t^{(N)} \colon |\alpha\rag \mapsto \left\lbrace \begin{matrix}
\Psi_t(|\alpha\rag) \text{ if } |\alpha| \leq N \\
0 \text{ otherwise.}
\end{matrix}\right.$$
The operator $\Psi_t^{(N)}$ is compact and as in the one dimensional case we have:
\begin{align*}
\|\Psi_t^{(N)} - \Psi_t\| &= \sup_{|\alpha|>N}\sup_{1\leq j \leq d}\left| \left(\frac{|\alpha|+2}{|\alpha|+1}\right)^{\ii t/2} - 1\right| \left(\frac{\alpha_j+1}{|\alpha|+1}\right)^{1/2}\\
					   &= \sup_{|\alpha|>N} \left| \left(\frac{|\alpha|+2}{|\alpha|+1}\right)^{\ii t/2} - 1\right|\\
					   &=\sup_{k>N} \left| \left(\frac{k+2}{k+1}\right)^{\ii t/2} - 1\right|
\end{align*}
The end of the proof is then identical to the one of Lemma \ref{WellDefd1}. We obtain, as before, that $\lim_{N \to \infty}\Psi_t^{(N)} = \Psi_t$ in norm. The operator $\Psi_t$ is thus compact, which proves the lemma.
\end{proof}

\begin{cor}We have the exact sequence:
$$\xymatrix{0\ar[r] & \CCC_0(\R^*_+,\mathcal{K}(\F^+(W))) \ar[r] & \T(W)\rtimes \R \ar[r] & \CCC_0(W^*\setminus\{0\}) \ar[r] & 0.}$$
\end{cor}
\begin{proof}
Similar as the one dimensional case. The crossed product preserves the exact sequence. The action is trivial on the quotient and the crossed product is trivial on the ideal of compact operators.
\end{proof}

\subsection{The number operator (plus one) on the Heisenberg group}

Let $(V,\omega)$ be a symplectic vector space of dimension $2d$. As before, we choose a compatible complex structure $J\in \End(V)$. Let $X_1,\cdots,X_d,Y_1,\cdots,Y_d$ a Darboux basis. We get a basis of $\heis(V,\omega)$ with $X_1,\cdots,X_d,Y_1,\cdots,Y_d,Z$. We say the Darboux basis is compatible with the complex structure $J$ if $\forall 1\leq j \leq d, JX_j = -Y_j, JY_j = X_j$. This is equivalent to $W_j = \frac{1}{\sqrt{2}}(X_j+ \ii Y_j), 1\leq j \leq d$ forming an orthonormal basis of $V^{1,0}$ for the hermitian structure given by $\ii\omega(\cdot, \overline{\cdot})$.

Let us consider the following (left-invariant) differential operators on $\Heis(V,\omega)$:
$$\Delta_{\pm} := -\frac{1}{2}\sum_{j = 1}^d X_j^2 + Y_j^2 \mp \ii \left(1-\frac{d}{2}\right) Z.$$

We can rewrite it in terms of the complex vectors $W_j, 1\leq j \leq d$. We have:

\begin{align*}
\Delta_+ &= -\sum_{j = 1}^d \overline{W_j}W_j - \ii (d+1) Z = -\sum_{j = 1}^d W_j\overline{W_j} - \ii Z,\\
\Delta_- &= -\sum_{j = 1}^d \overline{W_j}W_j + \ii Z = -\sum_{j = 1}^d W_j\overline{W_j} + \ii (d+1) Z.
\end{align*}

We have a graduation on $\heis(V,\omega)$ where $V$ is the first stratum and $\R Z$ the second one. This is compatible with the Lie bracket so we obtain a graduation on the universal envelopping algebra $\mathcal{U}(\heis(V,\omega)\otimes \CC)$. The later is isomorphic (as an algebra) to the algebra of (left-invariant) differential operators on $\Heis(V,\omega)$. We thus transport the graduation to the algebra of differential operators, giving a new notion of order. With this notion of order $\Delta_{\pm}$ are still of order $2$, but now all of its summands have order exactly 2.

\begin{prop}Given two Darboux bases for $V$: $X_1,\cdots,X_d,Y_1,\cdots,Y_d$ and $X'_1,\cdots,X'_d,Y'_1,\cdots,Y'_d$, let us denote by $\Delta_{\pm},\Delta'_{\pm}$ the respective operators. If the two bases are compatible with the same complex structure then $\Delta_{\pm}=\Delta'_{\pm}$.\end{prop}

\begin{lem}Let $M \in M_{2d}(\R)$. If $M$ is a symplectic matrix and commutes with the matrix $\begin{pmatrix}0 & I_d \\ -I_d & 0\end{pmatrix}$ then it is of the form:
$$M = \begin{pmatrix}A & B \\ -B & A\end{pmatrix}, \tsp A A+\tsp B B = I_d, \tsp A B = \tsp B A.$$
These conditions are equivalent to $A-\ii B \in \mathcal{U}_d(\CC)$.\end{lem}
\begin{proof}
The commutation with the complex structure gives $M = \begin{pmatrix}A & B \\ -B & A\end{pmatrix}$ for some $A,B \in M_d(\R)$ and the compatibility with the symplectic structure gives the two conditions on $A$ and $B$.
\end{proof}

\begin{proof}[Proof of the proposition]
Consider two Darboux basis compatible with the same structure $J$. The transition matrix then satisfies the condition of the lemma. The transition matrix from the basis $W_1,\cdots, W_d$ to $W'_1,\cdots,W'_d$ is given by $A-\ii B$ which is unitary. From this unitarity we easily see that $\sum_{j = 1}^d \overline{W_j}W_j = \sum_{j = 1}^d \overline{W'_j}W'_j$ thus $\Delta_{\pm} = \Delta'_{\pm}$.
\end{proof}

These differential operators can be considered as (one plus) the number operator on the Heisenberg group. Indeed, let $\lambda > 0$ and consider the representation on the symmetric Fock space $\pi_{\lambda} \in \widehat{\Heis(V,\omega)}$. We see $\pi_{\lambda}$ at the level of the Lie algebra (so a representation with anti-self-adjoint, unbounded operators) and extend it to the universal envelopping algebra. Remember that for $1\leq j \leq d$ we have $\pi(W_j) = \ii \sqrt{\lambda} A_j^*$ where the $A_j^*$ are the creation operators. We thus have:
$$\pi_{\lambda}(\Delta_+) = \lambda (N+1).$$
Similarly for $\lambda <0$:
$$\pi_{\lambda}(\Delta_-) = -\lambda (N+1).$$

\begin{rem}Although $\Delta_{\pm}$ had degree $2$ we get here a homogeneity of degree 1 in $\lambda$. This is explained by the fact that scaling the representation by a factor $\lambda >0$ scales the images of elements of $V$ (i.e. operators of order 1) by a factor $\sqrt{\lambda}$ as shown in Section \ref{HeisenbergIntroduction}.
\end{rem}

We can also compute the image of $\Delta_{\pm}$ in the other irreducible representations. For $\mu\in V^*$ we have:
$$\pi_{\mu}(\Delta_{\pm}) = \frac{1}{2}|\mu|^2.$$

We unfortunately have to use these two operators instead of just one (until Section \ref{Symbols technicalities} where we will glue the two). The problem is that they both have a bad behavior on the representations of the opposite sign. If $\lambda <0$ we have:
$$\pi_{\lambda}(\Delta_+) = -\lambda (N - d - 1),$$
and similarly for $\Delta_-$ in the positive representations. In particular they are not positive anymore and have a kernel ($\pi_{-1}(\Delta_+)$ vanishes on $\Sym^d(V^{0,1})$). We therefore need to consider these operators not on the whole Heisenberg $C^*$-algebra but on its quotients $I_{\pm}(V,\omega)$.

\subsection{Construction of the crossed product and isomorphism}

We will use the following results, whose proofs are postponed until Section \ref{Symbols technicalities} as they would disrupt the spirit of the rest of the proof, and need different tools that we have not introduced yet. As before, we fix a compatible complex structure \(J \colon V \to V\), and an orthonormal basis \(W_1,\cdots,W_d \in V^{1,0}\).

\begin{thm}\label{Toeplitz as Heisenberg multipliers}
    The operators $-\ii W_j\Delta_{+}^{-1/2}$ are densely defined on $I_+(V,\omega)$ and extend continuously to elements of $M(I_+(V,\omega))$. 
    The same result applies to $I_-(V,\omega)$ with the operators $-\ii \overline{W_j}\Delta_-^{-1/2}$.

    Equivalently, given a \(v\in V^{1,0}\) and an element \(f \in I_+(V,\omega)\), the continuous map \(\lambda > 0 \mapsto S_v\pi_{\lambda}(f)\) extends to an element of \(I_+(V,\omega)\). The continuous map on \(V^*\setminus \{0\}\cong \mathbb{S}^*V \times \R^*_+\) obtained at \(\lambda = 0\) is given by \((\tilde{v}\otimes 1)\pi_0(f)\) where \(\tilde{v}\in \CCC(\mathbb{S}^*V)\) is the map described in Remark \ref{QuotientToeplitz}.
\end{thm}

\begin{thm}\label{Family of unitary Heisenberg multipliers}
    The complex powers $\Delta_+^{\ii t/2}, t\in \R$ are densely defined on $I_+(V,\omega)$ and extend to unitary multipliers of $I_+(V,\omega)$. The resulting group of unitary multipliers is strongly continuous. The same result applies on \(I_-(V,\omega)\) with $\Delta_-^{\ii t/2}, t\in \R$.

    Equivalently, given a section \(f \in I_+(V,\omega)\), the continuous map \(\lambda > 0 \mapsto \lambda^{\ii t/2}(N+1)^{\ii t/2}\pi_{\lambda}(f)\) extends to an element of \(I_+(V,\omega)\). The continuous map on \(V^*\setminus \{0\}\cong \mathbb{S}^*V \times \R^*_+\) obtained at \(\lambda = 0\) is given by \(\mu\mapsto \frac{1}{2^{\ii t/2}}|\mu|^{\ii t} \pi_0(f)(\mu)\), where \(|\cdot|\) is the norm on \(V^*\) dual to the one induced by the \(\omega\) and \(J\) on \(V\).    
\end{thm}

With these two theorems, we can construct the desired map from the crossed product of the Toeplitz algebras to \(I_{\pm}(V,\omega)\) and get the desires isomorphism. 
Remember that there is an action \(\delta\) of the postitive real numbers on the Heisenberg group by inhomogeneous dilations. This action restricts to both \(I_{\pm}(V,\omega)\). We also want the map \(\mathcal{T}(V^{1,0})\rtimes \R \to I_+(V,\omega)\) to be \(\R^*_+\)-equivariant for the dual action on the left and the inhomogeneous dilations on the right. Let us state the main theorem of this paper.

\begin{thm}\label{Main Isom}Let \((V,\omega)\) be a symplectic vector space, endowed with a compatible complex structure. There are canonical $*$-homomorphisms: 
\[\T(V^{1,0}) \rtimes \R \to I_+(V,\omega), \ \T(V^{0,1}) \rtimes \R \to I_-(V,\omega),\] 
 that are isomorphisms of $C^*$-algebras. Moreover they are $\R^*_+$-equivariant for the dual action on the left side and the action induced by the inhomogeneous dilations $(\delta_{\lambda})_{\lambda>0}$ on the right side.
\end{thm}
\begin{proof}
    Let $\varphi \colon \T(V^{1,0}) \to M(I_+(V,\omega))$ be defined by $\varphi(S_j) = -\ii W_j\Delta_+^{-1/2}$. Let us show that this map is indeed well defined. We have:
    $$\forall \lambda> 0, \pi_{\lambda}(-\ii W_j\Delta_+^{-1/2}) = A_j^*(N+1)^{-1/2} = S_j.$$
    Therefore $\varphi$ is well defined on the (non-commutative) polynomials in $S_j, S_k^*, 1\leq j,k\leq d$. These are dense in $\T(V^{1,0})$. Now from the relation $\|x\|_{I_+(V,\omega)} = \sup_{t>0}\|x_t\|, s \in I_+(V,\omega)$, we see that the densely defined $\varphi$ is isometric and thus extends continuously to an isometry $\varphi \colon \T(V^{1,0}) \to M(I_+(V,\omega))$.
    Now for $x \in \T(V^{1,0})$ we have the relation:
    $$\forall t\in \R, \Ad(\Delta_+^{\ii t/2})\varphi(x) = \varphi(\Ad((N+1)^{\ii t/2})x).$$
    From this relation we get a $*$-homomorphism $\Phi \colon \T(V^{1,0}) \rtimes \R \to I_+(V,\omega)$. In a similar way we obtain a $*$-homomorphism $\T(V^{0,1}) \rtimes \R \to I_-(V,\omega)$.

    The map \(\Phi\) does not depend on the choice of orthonormal basis of \(V^{1,0}\). First, note that the operators \(\Delta_{\pm}\) do not depend on the choice of basis as discussed in the previous subsection.  Now, if \(v \in V^{1,0}\), then \(S_v\in \T(V^{1,0})\) is mapped by \(\varphi\) to \(-iv\Delta_{+}^{-1/2}\), this operator being understood in the same way as in Theorem \ref{Toeplitz as Heisenberg multipliers}. Thus \(\varphi\) and hence \(\Phi\) do not depend on any choice of basis for \(V^{1,0}\). 
    
    Let us show that \(\Phi\) is an isomorphism. From Theorem \ref{Toeplitz as Heisenberg multipliers} we see that commutators are mapped to \(0\) at \(t = 0\). In particular we get the commutative diagram:

    \[\xymatrix{0 \ar[r] &  \CCC_0(\R^*_+,\mathcal{K}) \ar[r] \ar[d] & \T(V^{1,0}) \rtimes \R \ar[r] \ar[d]^{\Phi} & \CCC_0(V^*\setminus \{0\}) \ar[r] \ar[d] & 0 \\
    0 \ar[r] &  \CCC_0(\R^*_+,\mathcal{K}) \ar[r] & I_+(V,\omega) \ar[r] & \CCC_0(V^*\setminus \{0\}) \ar[r] & 0.}\]
    
    The vertical arrows on the ideals and quotients are the identity (again using Theorem \ref{Toeplitz as Heisenberg multipliers}), in particular they are isomorphism. The rows are exact sequences so by the five lemma, \(\Phi\) is an isomorphism.

    We now check for the equivariance. In order to get it, we need the multipliers obtained from elements of the Toeplitz algebras to be invariant, which is the case. Indeed, we get constant sections of shift operators at non-zero time and the function obtained at \(t = 0\) is invariant under rescaling.
    We also need the relation:
    \[\forall \lambda > 0, \forall t\in \R, \delta_{\lambda}(\Delta_{\pm}^{\ii t/2}) = \lambda^{\ii t}.\]
    This condition is also satisfied\footnote{This condition was the reason for us to consider the conjugation by \(\Delta_{\pm}^{\ii t/2}\) and not with power \(\ii t\). This is only a matter of normalization of the duality between \(\R\) and \(\R^*_+\).}, so \(\Phi\) is \(\R^*_+\)-equivariant.
\end{proof}
\begin{cor}Let $(V,\omega)$ be a symplectic real vector space, endowed with a compatible complex structure. There is a canonical isomorphism:
$$\left(\T(V^{1,0})\bigoplus_{\CCC(\mathbb{S}^*V)}\T(V^{0,1})\right) \rtimes \R \to C^*_0(\Heis(V,\omega)).$$
This isomorphism is $\R^*_+$-equivariant for the dual action on the source space and the inhomogeneous dilations on the target space.
\end{cor}

\begin{proof}
    First, since the actions on the Toeplitz algebras agree on  the respective quotients, we get an isomorphism:
    \[\left(\T(V^{1,0})\bigoplus_{\CCC(\mathbb{S}^*V)}\T(V^{0,1})\right) \rtimes \R \cong \left(\T(V^{1,0})\rtimes \R\right)\bigoplus_{\CCC_0(V^*\setminus \{0\})}\left(\T(V^{0,1})\rtimes  \R \right).\]

    Let us denote by \(\mathbb{T}(V,\omega)\) this algebra. We can project from \(\mathbb{T}(V,\omega)\) on each component. We get the commutative diagram:
    \[\xymatrix{ & \T(V^{1,0})\rtimes \R \ar[r] & I_+(V,\omega) \ar[rd] & \\
                \mathbb{T}(V,\omega) \ar[ur] \ar[dr] & & & \CCC_0(V^*\setminus \{0\}). \\
                 & \T(V^{0,1}) \rtimes\R \ar[r] & I_-(V,\omega) \ar[ur] & }\]
We therefore get the desired isomorphism by property of the fibered product.
\end{proof}

\begin{cor}There is a canonical Morita equivalence 
\[\T(V^{1,0})\bigoplus_{\CCC(\mathbb{S}^*V)}\T(V^{0,1}) \sim C^*_0(\Heis(V,\omega))\rtimes_{\delta} \R^*_+.\]
\end{cor}

\begin{proof}
Using Takai duality we get isomorphisms
\begin{align*}
C^*_0(\Heis(V,\omega))\rtimes_{\delta} \R^*_+ &\cong \left(\T(V^{1,0})\bigoplus_{\CCC(\mathbb{S}^*V)}\T(V^{0,1})\right) \rtimes \R \rtimes \R^*_+ \\
                                              &\cong \left(\T(V^{1,0})\bigoplus_{\CCC(\mathbb{S}^*V)}\T(V^{0,1})\right) \otimes \mathcal{K}(L^2(\R)).
\end{align*}
The last line gives stable isomorphism hence Morita equivalence.
\end{proof}

\begin{rem}
    As mentioned in Remark \ref{Remarque:H ou N+1}, we could have chosen a different operator for the action and the group of multipliers, namely the quantum harmonic oscillator. As a left invariant differential operator on the Heisenberg group, it is defined as \(-\frac{1}{2}\sum_{i = 1}^d X_j^2 + Y_j^2\). On each representation \(\pi_{\lambda}, \lambda \neq 0\) it is mapped to the operator \(|\lambda|(N + \frac{d}{2})\). On the 1-dimensional representation \(\pi_{\mu}, \mu \in V^*\), it is mapped to \(\frac{1}{2}|\mu|^2\).
    As an unbounded multiplier of \(C^*_0(\Heis(V,\omega))\), it is positive. This is slightly better than the \(\Delta_{\pm}\), that do not individually have this property, which will force us to treat each of them on a restricted subset of representations and them glue back the result (see Section \ref{Symbols technicalities}).
    In any case, we need to understand the operators \(\Delta_{\pm}\) to construct the image of the shifts in the multipliers of \(I_{\pm}(V,\omega)\). Therefore, using them for the conjugation as well limits the number of operators to handle at the same time.
\end{rem}

\section{Symbols in the Heisenberg calculus and Weyl operators}\label{Symbols technicalities}

In this section, we give a meaning to the fractional and complex powers of some operators appearing in Section \ref{IsomorphismSection}. This uses Heisenberg pseudodifferential operators. This class of operators was introduced by Folland and Stein \cite{FollandSteinAnnonce,FollandSteinResultat} when studying the \(\partial_b\)-complex on complex manifolds with boundaries. One can construct a Hodge-like laplacian for this complex, it will however fail to be elliptic. The new pseudodifferential calculus allows to bypass this problem and find parametrices and estimates for this operator. It was later generalized in \cite{BealsGreiner} and subsequently by many other authors (e.g. \cite{EpsteinMelrose}, see also \cite{vanErpYuncken} for a more recent approach using the tangent groupoid).

The idea for our operators is that \(\Delta_{\pm}\) are not elliptic. Indeed, they are operators of order 2 and their principal symbol at the identity of the group is then equal to:
\[V^*\times \R \ni (\mu,t)\mapsto \frac{1}{2}|\mu|^2.\]
In particular, the symbol vanishes in the direction conormal to \(V\). To better understand these operators we need to consider \(Z\) as a differential operator of order 2. This makes sense because it is equal to \(X_1Y_1-Y_1X_1\) in the universal envelopping algebra. We thus consider a new filtration on differential operators. This filtration is compatible with the sum and product. Homogeneity now corresponds to the inhomogeneous dilations \((\delta_{\lambda})_{\lambda\in \R^*_+}\). 

To construct parametrices in this calculus, we need to construct a bigger class of operators: the pseudodiffrential operators. For what follows we only need to understand their symbols. The reader interrested in more generalities can look at the monographs \cite{BealsGreiner, EpsteinMelrose} or the survey \cite{Epstein}.

\begin{rem}
    Here we only consider operators on \(\Heis(V,\omega)\) that are left-invariant. In particular the symbols need only to be defined on \(\heis(V,\omega)^*\) instead of the whole cotangent bundle \(T^*\Heis(V,\omega)\). 
    The generalization to any Heisenberg (e.g. contact) manifold would however prove useful to better understand what happens in Section \ref{ContactExtension}. For the sake of simplicity, we limit our exposition to the case of \(\Heis(V,\omega)\). The only thing to keep in mind is that this serves as a local model for any contact manifold and that the operations we do here could be done locally and patched up on a contact manifold.

    We also consider only \emph{principal} symbols. This means that all our functions are homogeneous on the nose (and defined away from \(0\)). One can study a larger class of symbols by taking smooth functions satisfying some estimates at infinity and having an asymptotic expansion in terms of \(\delta\)-homogeneous functions. The principal part of such symbol is then the highest degree part in such expansion.
\end{rem}

\subsection{From Heisenberg symbols to Weyl operators}

A principal symbol of order \(m \in \mathbb{C}\) is a smooth function \(a \in \ \allowbreak \CCC^{\infty}\left(\heis(V,\omega)^*\setminus\{0\}\right)\) which is homogeneous of degree \(m\) for the inhomogeneous dilations \((\delta_{\lambda})_{\lambda>0}\). This means:
\[\forall \lambda>0, \forall (\mu,t)\in (V^*\times \R)\setminus\{0\}, a(\lambda\mu,\lambda^2t) = \lambda^m a(\mu,t).\]
We denote by \(\Sigma^m(\Heis(V,\omega))\) the space space of such functions. One can compose such symbols using the convolution product of the Heisenberg group\footnote{So one needs to consider the inverse Fourier transform of two symbols and take their convolution product and take the Fourier transform of the result. This product is noncommutative.}. Since the dilations are group homomorphisms, the composition of symbols gives product maps:
\[\Sigma^{m_1}(\Heis(V,\omega))\times \Sigma^{m_2}(\Heis(V,\omega))\to \Sigma^{m_1+m_2}(\Heis(V,\omega)), m_1,m_2\in\CC.\]

\begin{ex}
    We can understand \(\Delta_{\pm}\) as operators of order 2. Their principal symbol become respectively:
    \[V^* \times \R \ni (\mu,t) \mapsto \frac{1}{2}|\mu|^2 \pm t \left(1- \frac{d}{2}\right).\]
\end{ex}

Given the homogeneity of the symbols, we can restrict their studies to the following hyperplanes: \(V^*\times \{\pm 1\}\) and \(V^*\times \{0\}\). On the last the product of principal symbols is given by pointwise product. On the two others it is given by the Weyl products \(\#_{\pm}\):
\[\forall f,g \in \CCC^{\infty}(V^*), f\#_{\pm}g(v) := \int_{(x,y)\in V^*}e^{\pm 2\ii \omega^*(x,y)} f(v+x)g(v+y)\diff x \diff y.\]

Here to write the product we used the symplectic form \(\omega^*\) on \(V^*\) obtained by duality. We also have splitted \(V^*\) into complementary lagrangian subspaces (whence the \((x,y)\) coordinates). The product however does not depend on the choice of splitting.

The symbols obtained on \(V^*\times \{\pm 1\}\) behave themselves like operators. To understand that, we split \(V^* = L \oplus L'\) as a pair of complementary lagrangian subspaces. Consider then for a smooth function \(a\in \CCC^{\infty}(V^*)\) (with appropriate decay conditions) the Weyl quantization \cite{TaylorMicrolocal,TaylorHarmonic}:
\begin{align*}
    \op^W(a) \colon &\mathscr{S}(L) \to \mathscr{S}(L) \\
    &f \mapsto \left(x \mapsto \frac{1}{(2\pi)^d}\int\int e^{\ii \omega^*(x-y,\xi)}a\left(\frac{x+y}{2},\xi\right)f(y)\diff y \diff \xi\right)
\end{align*}

This operator makes sense and is continuous for smooth functions \(a\) having a 1-step polyhomogeneous asymptotic expansion in terms homogeneous functions on \(V^*\):
\[a \sim \sum_{k = 0}^{+\infty} a_{m-k}, \ a_{m-k}(\lambda\mu) = \lambda^{m-k} a_{m-k}(\mu), k\geq 0, \lambda>0, \mu\in V^*.\]
For a precise definition this expansion we refer to \cite{Guillemin,TaylorMicrolocal, EpsteinMelrose,Epstein,GorokhovskyvanErp}. The number \(m\) appearing in the expansion is the order of the symbol (or of the corresponding operator). We denote by \(\mathcal{W}^m(V^*,\omega^*)\) the set of such Weyl operators obtained from symbols of order \(m\) using the Weyl quantization. We obtain a pseudodifferential calculus: 
\begin{itemize}
    \item operators in \(\mathcal{W}^z(V^*,\omega^*)\) with \(\Re(z)\leq 0\) are bounded on \(L^2(L)\),
    \item operators in \(\mathcal{W}^z(V^*,\omega^*)\) with \(\Re(z) < 0\) are compact on \(L^2(L)\),
    \item if \(\op^W(a) \in \mathcal{W}^{z_1}(V^*,\omega^*), \op^W(b) \in \mathcal{W}^{z_2}(V^*,\omega^*)\) then we have
    \[\op^W(a)\op^W(b) = \op^W(a\#_+b) \in \mathcal{W}^{z_1+z_2}(V^*,\omega^*),\]
    \item if \(\op^W(a) \in \mathcal{W}^{z}(V^*,\omega^*)\) then we have \(\op^W(a)^T = \op^W(\bar{a}) \in \mathcal{W}^{\bar{z}}(V^*,\omega^*),\)
    \item if \(k\in \mathbb Z\), \(H \in \mathcal{W}^k(V^*,\omega^*)\) is a positive operator then its complex powers defined through functional calculus are also Weyl operators (with some holomorphic dependence, see \cite{Guillemin}):
    \[H^z \in \mathcal{W}^{kz}(V^*,\omega^*).\]
\end{itemize}

So far we have seen that we had maps:
\[\Sigma^m(\Heis(V,\omega))\to \mathcal{W}^m(V^*,\omega^*)\oplus\mathcal{W}^m(V^*,-\omega^*), \ m\in \CC,\]
 compatible with the products on both sides. This map cannot be surjective since the two symbols have to be glued on \(V^* \times \{0\}\subset \heis(V,\omega)^*\). If \(a \in \Sigma^m(\Heis(V,\omega))\) for \(m\in \CC\) then its image is a pair of Weyl symbols \((a_+,a_-) \in \mathcal{W}^m(V^*,\omega^*)\oplus\mathcal{W}^m(V^*,-\omega^*)\) for which we have 2-step polyhomogeneous expansions:

 \[a_+(v) \sim \sum_{k = 0}^{+\infty}a_{m-2k}\| v \|^{m-2j}, a_-(v) \sim \sum_{k = 0}^{+\infty}(-1)^ka_{m-2k}\| v \|^{m-2j}.\]
 Indeed, the two functions \(a_{\pm}\), considered as functions on the inhomogeneous sphere, intersected with \(V^*\times \R^*_{\pm} \) respectively, have to glue back on the equator (i.e. the intersection of the sphere with \(V^*\times \{0\}\)). This condition implies in particular that both Taylor expansions at any point of the equator have to agree. The equality of such Taylor expansions gives the polyhomogeneous expansions above.

 \begin{rem}
     This condition is however not enough. Indeed, the polyhomogeneous expansion above is not enough to recover the symbols \(a_{\pm}\). We can only recover them up to Schwartz functions on \(V^*\).
 \end{rem}

\subsection{\texorpdfstring{\(C^*\)}{C}-algebraic completions}

Since order 0 Weyl operators are bounded on \(L^2(L)\) (\(L\subset V^*\) being a fixed lagrangian subspace as before). We denote by \(W^0(V^*,\omega^*)\subset \mathcal{B}(L^2(L))\) the closure of \(\mathcal{W}^m(V^*,\omega^*)\). This closure does not depend on the choice of \(L\). Given a Weyl operator \(A \in \mathcal{W}^m(V^*,\omega^*)\), we can write it as \(A = \op^W(a)\). Keeping the leading term in the polyhomogeneous expansion of \(a\), we obtain a homogeneous function of order \(m\), i.e. a function on the sphere. This defines maps:
\[\sigma_m\colon \mathcal{W}^m(V^*,\omega^*) \to \CCC^{\infty}(\mathbb{S}^*V),\]
which are compatible with the products (the non-commutativity of symbols only appears with higher order terms).

For \(m = 0\) (or more generally \(\Re(m)\leq 0\)), this map extends continuously to the \(*\)-homomorphism:
\[\sigma_0 \colon W^0(V^*,\omega^*) \to \CCC(\mathbb{S}^*V).\]

For Heisenberg symbols, we can make them act by convolution (the one coming from the group law) on \(\mathscr{S}_0(\Heis(V,\omega))\), the space of smooth functions on \(\heis(V,\omega)^*\setminus \{0\}\) that have a Schwartz-type decay both at infinity and \(0\). This space is a Fréchet space, and for any \(\sigma \in \Sigma^m(\Heis(V,\omega)), m\in \CC\), the map:
\[\sigma \ast \colon \mathscr{S}_0(\Heis(V,\omega)) \to \mathscr{S}_0(\Heis(V,\omega))\]
is continuous.
The space \(\mathscr{S}_0(\Heis(V,\omega))\) can be seen as a subspace of \(C^*_0(\Heis(V,\omega))\) using Fourier transform. It is then a dense subspace.
For \(\Re(m) = 0\), the previous convolution map extends continuously to:
\[\sigma \ast \colon C^*_0(\Heis(V,\omega))\to C^*_0(\Heis(V,\omega)).\]
The same convolution can be considered on the right and we obtain that Heisenberg principal symbols of order \(0\) extend to multipliers of \(C^*_0(\Heis(V,\omega))\). Denote by \(\Sigma(\Heis(V,\omega))\subset M(C^*_0(\Heis(V,\omega)))\) the closure of the space of order \(0\) principal symbols in the multiplier algebra.

\begin{rem}
    Notice that taking \(\Re(m) = 0\) above is optimal. Indeed, if \(\Re(m)> 0\) we have a decay problem at infinity as usual. If \(\Re(m)<0\) however, we get a decay problem near \(0\in \heis(V^*,\omega^*)\).
\end{rem}

The following result is the key point that will help us prove Theorems \ref{Toeplitz as Heisenberg multipliers} and \ref{Family of unitary Heisenberg multipliers}.

\begin{thm}[\cite{EpsteinMelrose}]\label{Heisenberg Equal Two Weyl} The map \(\Sigma^0(\Heis(V,\omega)) \to \mathcal{W}^0(V^*,\omega^*) \oplus \mathcal{W}^0(V^*,-\omega^*)\) extends to an isomorphism of \(C^*\)-algebras:
  \[\Sigma(\Heis(V,\omega)) \to W^0(V^*,\omega^*) \bigoplus_{\CCC(\mathbb{S}V^*)} W^0(V^*,-\omega^*).\]
\end{thm}

\begin{cor}If we take \((Q_+,Q_-) \in W^0(V^*,\omega^*) \bigoplus_{\CCC(\mathbb{S}V^*)} W^0(V^*,-\omega^*)\) and \(f = (f_t)_{t\in\R} \in C^*_0(\Heis(V,\omega))\). Then, the sections
  \((Q_{\pm}f_{\pm t})_{t>0}\) extend continuously at \(t = 0\) to elements of \(I_{\pm}(V,\omega)\) respectively. Their extension at \(t=0\) coincide and are equal to \(\sigma^0(Q_{\pm})f_0\) (here the product is the commutative product of functions).
\end{cor}

\begin{cor}
    If \(Q\in W^0(V^*,\omega^*)\) and \(f = (f_t)_{t\geq 0} \in I_+(V,\omega)\) then \((Qf_t)_{t>0}\) extends to a unique element of \(I_+(V,\omega))\) with value at 0 equal to \(\sigma_0(Q)f_0\). A similar statement holds for \(W^0(V^*,-\omega^*)\) and \(I_-(V,\omega)\). Similar results hold for operators of order \(\ii t, t\in \R\).
    
    These results define *-homomorphisms \(W^0(V^*,\pm\omega^*)\to M\left(I_{\pm}(V,\omega)\right)\).
\end{cor}

\begin{proof}
    We can just consider \((Q,Q) \in \Sigma^0(\Heis(V,\omega))\) and the aforementioned map \(\Sigma^0(\Heis(V,\omega)) \to M(C^*_0(\Heis(V,\omega)))\). One can then see that this action restricts to \(I_+(V,\omega)\).
\end{proof}

\subsection{The shifts as Heisenberg operators}

We now prove Theorems \ref{Toeplitz as Heisenberg multipliers} and \ref{Family of unitary Heisenberg multipliers}. 
The proof goes as follows: we realize the operators mentioned in both theorems as symbols in the Heisenberg calculus, of order \(0\) for the images of  d-shifts, or order \(\ii t\) for the unitary multipliers. 
Using the continuity of such operators on \(C^*_0(\Heis(V,\omega))\) we get the result. 

Once again we fix a compatible complex structure on \((V,\omega)\) and fix an orthonormal basis \(W_1,\cdots,W_d \in V^{1,0}\). From this we get the left-invariant differential operators \(\Delta_{\pm}\) on \(\Heis(V,\omega)\). Write \(\Delta_{+,+},\Delta_{-,+} \in \mathcal{W}^2(V^*,\omega^*)\) their restrictions on the upper hemisphere and \(\Delta_{+,-},\Delta_{-,-}\) on the lower hemisphere. 
The operators \(\Delta_{+,+},\Delta_{-,-}\) are elliptic, their principal symbol is \(\mu \mapsto \frac{1}{2}|\mu|^2\).
As operators they are positive (so more than elliptic, they are even invertible). Indeed by construction, they are both mapped to the number operator plus 1 in the Bargmann-Fock representation. We can thus consider their complex powers. 

On the one hand, we get order \(-1\) Weyl operators \(\Delta_{+,+}^{-1/2}, \Delta_{-,-}^{-1/2}\). We then get the operators \(-\ii W_j\Delta_{+,+}^{-1/2} \in \mathcal{W}^0(V^*,\omega^*)\) and \(-\ii \overline{W_j}\Delta_{-,-}^{-1/2} \in \mathcal{W}^0(V^*,-\omega^*)\). These act continuously on \(I_+(V,\omega)\) and \(I_-(V,\omega)\) as multipliers. This proves Theorem \ref{Toeplitz as Heisenberg multipliers}.

On the other hand, we can take the complex powers \(\Delta_{+,+}^{\ii t /2}\) and \(\Delta_{-,-}^{\ii t /2}\). They act as a continuous unitary multipliers of \(C^*_0(\Heis(V,\omega))\). This proves Theorem \ref{Family of unitary Heisenberg multipliers}.

These considerations allow us to reinterpret some previous results. 

\begin{prop}
    Let \((V,\omega)\) be a symplectic vector space with a compatible complex structure. We have natural isomorphisms of \(C^*\)-algebras:
    \begin{align*}
        \T(V^{1,0}) &\isomto W^0(V^*,\omega^*) \\
        \T(V^{0,1}) &\isomto W^0(V^*,-\omega^*).
    \end{align*}
    Combining the two we obtain:
        \[\T(V^{1,0}) \bigoplus_{\CCC_0(\mathbb{S}^*V)} \T(V^{0,1}) \isomto \Sigma(\Heis(V,\omega)).\]
    
\end{prop}
\begin{proof}
    The map \(\T(V^{1,0}) \isomto W^0(V^*,\omega^*)\) is constructed by mapping the shifts \(S_j\) to \(-\ii W_j\Delta_{+,+}^{-1/2}\) for \(1\leq j \leq d\). This map extends to polynomial expressions in the \(S_j, S_k^*, 1\leq j,k\leq d\) and is an isometry. Indeed if we compose the map \(W^0(V^*,\omega^*)\) with the representation corresponding to the fiber at \(t=1\) of \(I_+(V,\omega)\) we see that \(S_j\) is mapped to itself. Therefore the map defined on polynomials in the \(S_j, S_k^*\) extends to a *-homomorphism.
    
    Now the morphism preserves the respective exact sequences. This means that the map \(\T(V^{1,0}) \to \CCC(\mathbb{S}^*{V^{1,0}})\) corresponds to \(\sigma_0\) for \(W^0(V^*,\omega^*)\). This is under the identification of complex vector spaces \((V,J) \cong V^{1,0}\) yielding \(\mathbb{S}^*{V^{1,0}} \cong \mathbb{S}^*V\). If we look at the map between the respective exact sequences, we have an isomorphism on the ideal and the quotient, thus in the middle as well.

    The same goes for \(\T(V^{0,1})\) with the small adjustment that now \(V^{0,1} \cong (V,-J) = \overline{(V,J)}\).

    Using these identifications we can build the algebra 
    \[\T(V^{1,0}) \bigoplus_{\CCC_0(\mathbb{S}^*V)} \T(V^{0,1}).\] 
    Using the two previous morphisms, this algebra is isomorphic to 
    \[W^0(V^*,\omega^*) \bigoplus_{\CCC(\mathbb{S}V^*)} W^0(V^*,-\omega^*),\] 
    hence to \(\Sigma(\Heis(V,\omega))\) using Theorem \ref{Heisenberg Equal Two Weyl}.
\end{proof}

We can give a non-commutative geometric interpretation of this result. The Toeplitz algebra \(\T(V^{1,0})\) can be seen as a quantized/non-commutative analog of the closed unit ball in \(V^{1,0}\). Indeed, its spectrum is a quotient of the that ball, where all the points in the open ball are all equivalent and any point on the boundary remains untouched.

Similarly, we can see \(\Sigma(\Heis(V,\omega))\) as a non-commutative analog of the sphere in \(\Heis(V,\omega)\) for the inhomogeneous dilations. 

The isomorphism \(\T(V^{1,0}) \bigoplus_{\CCC_0(\mathbb{S}^*V)} \T(V^{0,1}) \isomto \Sigma(\Heis(V,\omega))\) is then the non-commutative analog of the fact that a sphere can be obtained by gluing two closed balls along their boundary.

The isomorphism \(\left(\T(V^{1,0})\bigoplus_{\CCC(\mathbb{S}^*V)}\T(V^{0,1})\right) \rtimes \R \to C^*_0(\Heis(V,\omega))\) can then be seen as the non-commutative analog of the fact that taking the quotient by dilations of a vector space minus its 0 element gives the sphere of that space.

\subsection{A more general result}

Theorem \ref{Main Isom} is actually a reformulation, using Toeplitz operators, of a more general result on graded nilpotent groups.

A graded (nilpotent) Lie algebra is a (real or complex) Lie algebra \(\g\) that decomposes as a finite sum of vector subspaces \(\g = \g_1 \oplus \cdots \oplus \g_r\) with the property:
\[\forall 1\leq i,j \leq r, [\g_i,\g_j] \subset \g_{i+j},\]
with the convention \(\g_j = 0\) for any \(j > r\). Such a Lie algebra is automatically nilpotent (with step at most \(r\)).

A graded Lie group is a Lie group whose Lie algebra is graded. Here we will consider \(G\), a connected, simply connected, graded Lie group.

\begin{ex}
    Abelian groups are graded with \(\g_1 = \g, r = 1\).

    The Heisenberg group \(\Heis(V,\omega)\) is graded with \(\g_1 = V, \g_2 = \R, r = 2\).
\end{ex}

On the Lie algebra, we can define a family of inhomegeneous dilations \((\delta_{\lambda})_{\lambda > 0}\) by taking \(\lambda^j\Id\) on \(\g_j\). These maps preserve the Lie algebra structure so we can lift them to group homomorphisms.

We can use these functions to define spaces of homogeneous distributions \(\Sigma^m(G),\) \(m\in \CC\) as we did for the Heisenberg group. Since the dilations preserve the group structure, they are compatible with the group convolution and thus we get composition maps:
\[\Sigma^m(G) \times \Sigma^n(G) \to \Sigma^{m+n}(G), m,n\in \CC.\]

Elements of the universal envelopping algebra \(\mathcal{U}(\g)\) correspond to the differential operators for that calculus (i.e. certain distributions of integer order). We can endow this algebra with a grading using the one on \(\g\). The algebras \(\Sigma^m(G)\) only involve homogeneous elements so these correspond to elements in \(\mathcal{U}(\g)\) of pure degree.

These spaces of distributions satisfy similar properties as the ones described for \(G = \Heis(V,\omega)\), see \cite{CGGP}.

Most importantly, elements of \(\Sigma^{\ii t}(G), t\in \R\) act as multipliers of \(C^*_0(G)\). In particular we can consider the \(C^*\)-algebra \(\Sigma(G)\) obtained by taking the closure of order \(0\) symbols in the multiplier algebra. We can then phrase a generalization of Theorem \ref{Main Isom}.

\begin{thm}[\cite{Ewert,Cren}]
    Let \(G\) be a connected, simply connected, graded group. Let \(\Delta \in \mathcal{U}(\g)\) be an element of pure degree seen as a positive element in \(\Sigma^m(G)\) for some \(m\in \N\). 
    Consider the conjugation action by the \(\Delta^{\ii t/m}, t\in \R\) on \(\Sigma(G)\). Then there is an isomorphism of \(C^*\)-algebras:
    \[\Sigma(G)\rtimes \R \isomto C^*_0(G).\]
    This isomorphism is equivariant for the dual action on the left hand side and the action of the inhomogeneous dilations on the right hand side.
\end{thm}

For the Heisenberg group, we were thus able to decompose the symbol algebra \(\Sigma(G)\) into two copies of algebras easier to understand. I do not know if a similar splitting is possible for more general graded groups (with an easy understanding of the summands).

The algebra \(\Sigma(G)\) can also be seen as a non-commutative analog of functions on the sphere of \(G\) for the inhomogeneous dilations. The last theorem is then a generalization of the interpretation we made for the Heisenberg group, that the quotient by dilations of a vector space (minus 0) is the sphere. What we are asking for now is a potential Toeplitz-like model for the "hemispheres" of \(\Sigma(G)\). 

\section{Continuous version on contact manifolds}\label{ContactExtension}

\subsection{Continuous field of Toeplitz algebras}

Let $(M,H)$ be a contact manifold. We assume that the contact structure is co-oriented, i.e. the bundle $\faktor{TM}{H}$ is oriented. We can therefore choose a contact form $\theta \in \Omega^1(M)$ such that $\ker(\theta) = H$ and $\diff\theta_{|H}$ is a symplectic form on each fiber.

Let $J\colon H \to H$ be an almost complex structure compatible with the symplectic structure on $H$. Let $H^{1,0}\oplus H^{0,1} = H\otimes \CC$ be the corresponding decomposition of the complexified bundle. 

From this we build a bundle of Toeplitz algebras $\T(M,H)$ over $M$. To do so we consider the Hilbert bundle $\mathcal{F}^+(H^{1,0})$ over $M$ as previously.

Locally, over some trivializing $U \subset M$, we can fix an orthonormal frame bundle $W_1,\cdots,W_n\in \Gamma(U,H^{1,0})$. From these sections, we get shift operators on each fiber $S_1,\cdots,S_d$. Now if $X\in \Gamma(U,H^{1,0})$ we can write it as $X = \sum_{j = 1}^d f_jW_j$ with $f_j\in \CCC^{\infty}(U)$. Then the shift operator by $X$, 
\[S_X := \Sym(X\otimes \cdot ) \colon \Gamma(U,\mathcal{F}^+(H^{1,0}))\to \Gamma(U,\mathcal{F}^+(H^{1,0})),\]
becomes $S_X = \sum_{j = 1}^d f_jS_j$. Therefore the Toeplitz algebra over $U$ does not depend on the choice of local frame bundle. Now if $X \in \mathfrak{X}(M)$, the operator $S_X$ is bounded on each fiber but the operator norm might blow up at infinity in $M$ if $M$ is not compact. However if $X$ is compactly supported then $S_X$ extends to a bounded operator on $L^2(M,\mathcal{F}^+(H^{1,0}))$ with norm:
$$\|S_X\| = \sup_{m \in M}\|S_{X(m)}\|_{\mathcal{B}(\mathcal{F}^+(H_m^{1,0}))}.$$
We can then set:
$$\T(H^{1,0}) = C^*(S_X, X \in \mathfrak{X}_c(M)) \subset \mathcal{B}(L^2(M,\mathcal{F}^+(H^{1,0}))).$$

\begin{ex}If $X \in \Gamma(M,H^{1,0})$ goes to $0$ at infinity then $S_X \in \T(H^{1,0})$.\end{ex}

\begin{prop}$\T(H^{1,0})$ is the algebra of continuous sections vanishing at infinity of the bundle of Toeplitz algebras \(\left(\mathcal{T}(H_m^{1,0})\right)_{m\in M}\), the bundle structure being described above.\end{prop}
\begin{proof}
This follows from the fact that we can cover $M$ by Darboux coordinate charts where all the bundles involved are trivialized and that if \(X \in \Gamma_c(M,H^{1,0})\) then \(\|S_X\| = \|X\|\) so necessarily the sections have to vanish at infinity.
\end{proof}

\begin{cor}We have the exact sequence:
\[\xymatrix{0 \ar[r] & \CCC_0(M)\otimes\mathcal{K} \ar[r] & \T(H^{1,0}) \ar[r] & \CCC_0(\mathbb{S}^*H) \ar[r] & 0}.\]
\end{cor}

Likewise, we can take the opposite complex structure and consider the algebra \(\T(H^{0,1})\).

\begin{rem}[The non-co-oriented case]\label{Not co-oriented}  Assume now that the contact structure is possibly not-co-oriented. We can construct a 2-sheet cover $\pi \colon \tilde{M}\to M$ such that $\tilde{H}:= \pi^*H$ is a co-oriented contact structure (this cover is trivial if and only if $H$ is already co-oriented). Let us choose a contact form $\theta \in \Omega^1(\tilde{M})$ and a complex structure on $\tilde{H}$ compatible with the symplectic structure induced by $\theta$. We get a bundle of Toeplitz algebras $\T(\tilde{H}^{1,0})$. Now $\Z_2$ acts on $\tilde{M}$ by deck transformations. If $\gamma$ denotes its non-trivial element then $\gamma^*\tilde{H} = \tilde{H}$, however the symplectic structures induced by $\gamma^*\diff\theta$ and $\diff\theta$ are opposite to one another, i.e. they have opposite orientation. Therefore at the level of the Toeplitz algebras we have 
\[\gamma^* \T(\tilde{H}^{1,0}) = \T(\tilde{H}^{0,1}) = \T(\tilde{H}^{1,0})^{op},\]
so we cannot obtain a bundle of Toeplitz algebras on $M$ because of this orientation ambiguity. We will see however that the fibered pair of Toeplitz algebras still make sense in the non-co-oriented case.
\end{rem}

As in Section \ref{IsomorphismSection}, we can consider the number operator, this time on each fiber. Write \(\mathcal{N}\) for the diagonal operator on sections of \(\mathcal{F}^+(H^{1,0})\) equal to \(k\Id\) on the subspace of sections of \(\Sym^k(H^{1,0})\). The operator \(\mathcal{N}+1\) is then unbounded and positive, we can consider its complex powers.

The action \(\Ad((\mathcal{N}+1)^{\ii t/2})\) is again well defined on \(\mathcal{B}(L^2(M,\mathcal{F}^+(H^{1,0})))\) and preserves \(\T(H^{1,0})\). This is true fiberwise and if the norm of a section goes to \(0\) at infinity then this is still the case for its conjugation by \((\mathcal{N}+1)^{\ii t/2}\).

We can thus take the crossed product by the \(\R\)-action and again have an exact sequence:
\[\xymatrix{0\ar[r] & \CCC_0(M\times \R^*_+,\mathcal{K}) \ar[r] & \T(H^{1,0}) \rtimes \R \ar[r] & \CCC_0(H^*\setminus M) \ar[r] & 0,}\]
with $M\subset H^*$ being seen as the zero section.

We can also take the opposite algebra and fiber both over their common quotient. As before we get:

\[\xymatrix@C-1.25pc{0\ar[r] & \CCC_0(M\times \R^*,\mathcal{K}) \ar[r] & \left(\T(H^{1,0}) \bigoplus_{\CCC_0(\mathbb{S}^*H)} \T(H^{0,1}) \right)\rtimes \R \ar[r] & \CCC_0(H^*\setminus M) \ar[r] & 0.}\]

\subsection{Continous field of Heisenberg groups and isomorphism}

We can thus construct a bundle of Heisenberg Lie algebras on $M$ the following way: write $\ttt_HM = H\oplus \R$ and define a Lie bracket on each fiber by:
\[\forall X,Y \in H_x, [X,Y] = \diff_x\theta(X,Y).\]
This defines a bundle of Lie algebras that we can integrate using the Baker-Campbell-Hausdorff formula on each fiber. We obtain a locally trivial (e.g. in Darboux coordinates) bundle of groups denoted by $T_HM$, with fiber at \(x \in M\) given by \(\Heis(H_x, \diff_x\theta)\). The (groupoid) $C^*$-algebra of $T_HM$ is then the algebra of sections, vanishing at infinity, of a locally trivial bundle of $C^*$-algebras over $M$, with fiber over a point $x\in M$ equal to the group \(C^*\)-algebra $C^*(\Heis(H_x,\diff_x\theta))$. We can see it as a completion of the algebra of compactly supported smooth sections $\Gamma_c(M,T_HM)$ with the convolution defined fiberwise using the Heisenberg group structure.

Like in the group case, we can take on each fiber the kernel of the trivial representation of the fiber. Denote by $C^*_0(T_HM)$ the corresponding ideal (it is still a locally trivial bundle of $C^*$-algebras over $M$). We have the exact sequence:

$$\xymatrix{0 \ar[r] & \CCC_0(M)\otimes\CCC_0(\R^*,\mathcal{K}) \ar[r] & C^*_0(T_HM) \ar[r] & \CCC_0(H^*\setminus M) \ar[r] & 0.}$$

\begin{rem}
   Here we used the 1-form \(\theta\) to define the bundle of Heisenberg groups. We actually don't need it. We can equivalently set \(\ttt_HM = H \oplus \faktor{TM}{H}\) and define the Lie bracket of sections as:
    \[\forall X,Y\in \Gamma(M,H), [X,Y]_{\ttt_HM} := [X,Y]\mod \Gamma(H).\]
    This defines a Lie algebroid structure on \(\ttt_HM\) (with trivial anchor). This structure localizes on each fiber, i.e. \(\ttt_HM\) is a smooth family of Lie algebras over \(M\). Trivializing the line bundle \(\faktor{TM}{H}\) gives an isomorphism with the other \(\ttt_HM\) as defined previously. In particular this new definition is still valid when the contact structure is not co-oriented. 
\end{rem}

We can extend the notion of inhomogeneous dilations to this case. On the vector bundle \(H \oplus \R\) they are of the form \(\delta_{\lambda} = \lambda \oplus \lambda^2, \lambda > 0\). These maps preserve the Lie algebra structure on each fiber and thus lift to maps on \(T_HM\) preserving the fiberwise group structure (i.e. groupoid automorphisms).

In the co-oriented case, we can globally differentiate between positively indexed and negatively indexed representations in the spectrum by using the image of the Reeb field. The Reeb field is the unique vector field \(Z\in \mathfrak{X}(M)\) such that \(\theta(X) = 1\) and \(\diff\theta(Z,\cdot) = 0\). This vector field corresponds in each fiber to the element \((0,1)\in H_x \times \R\). If \(\pi\) is one of the infinite dimensional representations of the fiber \(\Heis(H_x,\diff_x\theta)\) then \(\diff\pi(Z) = \ii \lambda \Id, \lambda \neq 0\). We can thus globally distinguish between positive and negative representations in each fiber by looking at the image of the Reeb field.

Using this distinction we can again split the \(C^*\)-algebra into two parts:
\[C^*_0(T_HM) = I_+(H,\diff\theta) \bigoplus_{\CCC_0(H^*\setminus M)} I_-(H,\diff\theta).\]
The algebra \(I_{\pm}(H,\diff\theta)\) is the algebra of continuous sections vanishing at infinity of the field \((I_{\pm}(H_x,\diff_x\theta))_{x\in M}\). Both of these algebras are stable under the action of the inhomogeneous dilations.

\subsection{Isomorphism on compact contact manifolds}

Assume that \((M,H)\) is co-oriented, and fix a \(\theta\in \Omega^1(M)\) with \(\ker(\theta) = H\) and \(\diff\theta\) symplectic on \(H\). We also choose a compatible complex structure on \((H,\diff\theta)\) As before we get Toeplitz algebras \(\T(H^{1,0}), \T(H^{0,1})\).

We can define the symbols spaces \(\Sigma^m(T_HM) \) of Section \ref{Symbols technicalities} in a similar way (see e.g. \cite{Epstein,vanErpYuncken}). They are the smooth sections of the bundles \(\left(\Sigma^m(\Heis(H_x,\diff_x\theta))\right)_{x\in M}\) with compact support. Taking the completion \(\Sigma(T_HM)\) for \(m = 0\) thus gives an algebra of sections vanishing at infinity. 

Even though the constructions in Sections \ref{IsomorphismSection} and \ref{Symbols technicalities} could be made fiberwise, in order to get an isomorphism, they would need to be globally defined. In particular, one would need to work with a global version of the operators \(\Delta_{\pm}\) and understand their behaviour at infinity. To avoid these technicalities, we now restrict ourselves to \(M\) \emph{compact}.

\begin{rem}
    In the non-compact case, write \(\Sigma_b(T_HM)\) for the bounded sections of the bundle \((\Sigma(\Heis(H_x,\diff_x\theta)))_{x\in M}\). This is exactly the multiplier algebra of the algebra \(\Sigma(T_HM)\) of sections vanishing at infinity of the same bundle. Locally we can send sections over some open subset \(U\) to a section of multipliers of \(C^*_0(T_HM)_{|U}\). One could use this fact get the desired isomorphism in the non-compact case. However, the images of the shifts were produced using the inverse square-root of \(\Delta_{\pm}\). It is then not clear that there is a choice of \(\Delta_{\pm}\) for which these square-root inverses would be bounded at infinity.
\end{rem}

In local Darboux coordinates over a trivializing set \(U\subset M\) we get a \(\CCC^{\infty}(U)\)-basis of sections of \(H_{|U}\), \(X_1,\cdots,X_d,Y_1,\cdots,Y_d\). We can moreover choose such coordinates so that \(\forall 1\leq j \leq d, JX_j = Y_j\). We then have 
\[\forall 1\leq j,k\leq d, [X_j,Y_k] = \delta_{j,k}Z,\]
where \(Z\) is the globally defined Reeb field (it is globally defined because the contact structure is co-oriented). On the other hand, the \(X_j\)'s commute between them and so do the \(Y_j\)'s.
With these vector fields, we can construct, locally, operators \(\Delta_{\pm}\) with the same formula as on the Heisenberg group. Operators \(\Delta_{\pm}\in \Sigma^2(T_HM)\) are then constructed using partitions of unity.

\begin{rem}
    We can view this local construction of \(\Delta_{\pm}\) in another way. Choosing Darboux coordinates is the same thing as choosing a local contactomorphism with an open subset of the Heisenberg group. We can then take the operators \(\Delta_{\pm}\) defined on the Heisenberg group and pull them back to our coordinate chart.
\end{rem}

The same constructions as in Section \ref{Symbols technicalities} carry on in this new setting. If \(X \in \Gamma(M,H^{1,0})\) then \(-\ii X(\Delta_+)^{-1/2} \in M(I_+(H,\diff\theta))\).  Likewise we have a strongly continuous family \(\Delta_+^{\ii t/2} \in M_u(I_+(H,\diff\theta))\) like in Theorem \ref{Family of unitary Heisenberg multipliers}.

From this, we obtain a map \(\T(H^{1,0})\rtimes \R \to I_+(H,\diff\theta)\), which restricts over each point of \(M\) to the map defined in Theorem \ref{Toeplitz as Heisenberg multipliers} between the respective fibers. Similar arguments give a map \(\T(H^{0,1})\rtimes \R \to I_-(H,\diff\theta)\). We thus get the same conclusions as in Section \ref{IsomorphismSection} using the fact that these maps realise an isomorphism fiberwise.

\begin{thm}\label{Isom Contact}
    Let \(M\) be a compact manifold with contact structure \(H\) co-oriented by \(\theta\in\Omega^1(M)\) and a fixed complex structure \(J\in\End(H)\) compatible with \(\diff\theta\). Then there are isomorphisms of \(C^*\)-algebras:
    \[\T(H^{1,0})\rtimes \R \isomto I_+(H,\diff\theta), \T(H^{0,1})\rtimes \R \isomto I_-(H,\diff\theta).\]
    These maps glue to an isomorphism of \(C^*\)-algebras:
    \[\left(\T(H^{1,0}) \bigoplus_{\CCC(\mathbb{S}^*H)} \T(H^{0,1}) \right)\rtimes \R \isomto C^*_0(T_HM).\]
    All these isomorphisms are \(\R^*_+\)-equivariant for the dual action on the left hand side and the inhomogeneous dilations on the right hand side.
\end{thm}

\begin{cor}
    There is a Morita equivalence:
    \[\T(H^{1,0}) \bigoplus_{\CCC(\mathbb{S}^*H)} \T(H^{0,1}) \sim C^*_0(T_HM)\rtimes \R^*_+.\]
\end{cor}

Assume now that the contact manifold \(M\) is still compact but the contact structure is not co-oriented. As in Remark \ref{Not co-oriented}, we can consider a 2-sheet covering \(p \colon \tilde{M}\to M\) such that the pullback \(\tilde{H} = p^*H\) on \(\tilde{M}\) is co-oriented. Then, we have a \(\R^*_+\)-equivariant isomorphism:
\[C^*_0\left(T_{\tilde{H}}\tilde{M}\right) \cong \left(\T\left(\tilde{H}^{1,0}\right) \bigoplus_{\CCC(\mathbb{S}^*\tilde{H})} \T\left(\tilde{H}^{0,1}\right)\right)\rtimes\R. \]
Now we have \(T_{\tilde{H}}\tilde{M} \cong p^*T_HM\), and the action of \(\faktor{\Z}{2\Z}\) on \(\tilde{M}\) by deck transformations preserves \(C^*_0\left(T_{\tilde{H}}\tilde{M}\right)\). The non-trivial deck transformation exchanges the roles of \(\tilde{H}^{1,0}\) and \(\tilde{H}^{0,1}\). It also sends the Reeb field to its opposite, hence sends \(\Delta_+\) to \(\Delta_-\). Therefore we get:

\begin{thm}
    Let \(M\) be a compact contact manifold with contact structure \(H\). Let \(\tilde{H}\) be the co-oriented contact structure on the 2-sheet covering of \(M\) as above. We fix a contact 1-form \(\theta\in \Omega^1(M)\) and a compatible complex structure \(J\) on \(\tilde{H}\). There is a \(\R^*_+\)-equivariant isomorphism of \(C^*\)-algebras:
    \begin{align*}
        C^*_0(T_HM) &\cong \left(\left(\T\left(\tilde{H}^{1,0}\right)                \bigoplus_{\CCC(\mathbb{S}^*\tilde{H})} \T\left(\tilde{H}^{0,1}\right)\right)\rtimes\R\right)^{\faktor{\Z}{2\Z}}\\
                    &\cong \left(\T\left(\tilde{H}^{1,0}\right) \bigoplus_{\CCC(\mathbb{S}^*\tilde{H})} \T\left(\tilde{H}^{0,1}\right)\right)^{\faktor{\Z}{2\Z}}\rtimes\R
    \end{align*}
    If the contact structure is co-oriented we can push forward \(J, \theta\) to \(M\) and we recover the isomorphism of Theorem \ref{Isom Contact}. We also get the Morita equivalence:
    \[C^*_0(T_HM) \rtimes\R^*_+ \sim \left(\T\left(\tilde{H}^{1,0}\right) \bigoplus_{\CCC(\mathbb{S}^*\tilde{H})} \T\left(\tilde{H}^{0,1}\right)\right)^{\faktor{\Z}{2\Z}}\]
\end{thm}

\nocite{*}
\bibliographystyle{plain}
\bibliography{Biblio}
\end{document}